\documentclass{amsart}

\newcommand{\cE}{\mathcal{E}}

\newcommand{\cM}{\mathcal{M}}
\newcommand{\cO}{\mathcal{O}}\newcommand{\cP}{\mathcal{P}}

\newcommand{\cU}{\mathcal{U}}
\newcommand{\cX}{\mathcal{X}}
\newcommand{\cY}{\mathcal{Y}}

\newcommand{\bI}{\mathbb{I}}
\newcommand{\bL}{\mathbb{L}}
\newcommand{\bN}{\mathbb{N}}

\newcommand{\bR}{\mathbb{R}}

\usepackage{lmodern}
\usepackage{caption}
\usepackage[dvipsnames,svgnames,x11names,hyperref]{xcolor}
\usepackage{url,graphicx,verbatim,amssymb,enumerate,stmaryrd,mathtools,bbm}
\usepackage[pagebackref,colorlinks,citecolor=blue,linkcolor=blue,urlcolor=blue,filecolor=blue]{hyperref}
\usepackage{tikz,tikz-cd} 
\usetikzlibrary{arrows,decorations.pathmorphing,decorations.pathreplacing,backgrounds,fit,positioning,shapes.symbols,chains,calc}
\usepackage{microtype}

\newtheorem{theorem}{Theorem}[section]
\newtheorem{lemma}[theorem]{Lemma}
\newtheorem{proposition}[theorem]{Proposition}
\newtheorem{corollary}[theorem]{Corollary}

\theoremstyle{definition}
\newtheorem{definition}[theorem]{Definition}

\newtheorem{assumption}[theorem]{Assumption}
\theoremstyle{remark}
\newtheorem{example}[theorem]{Example}
\newtheorem{remark}[theorem]{Remark}

\newcommand{\cat}[1]{\mathsf{#1}}

\newcommand{\mr}[1]{{\rm #1}}
\newcommand{\fS}{\mathfrak{S}}
\newcommand{\I}{\mathbb I}
\newcommand{\lra}{\longrightarrow}

\newcommand{\bfE}{\mathbf{E}}
\newcommand{\bfF}{\mathbf{F}}

\newcommand{\bfT}{\mathbf{T}}
\newcommand{\bfA}{\mathbf{A}}
\newcommand{\bfB}{\mathbf{B}}
\newcommand{\bunit}{\mathbbm{1}}

\title{The May-Milgram filtration and $\cE_k$-cells}
\author{Inbar Klang}
\email{klang@math.columbia.edu}
\address{Department of Mathematics \\
2990 Broadway\\
New York NY, 10027\\ USA}

\author{Alexander Kupers}
\email{kupers@math.harvard.edu}
\address{Department of Mathematics \\
	One Oxford Street \\
	Cambridge MA, 02138 \\USA}
\thanks{Alexander Kupers is supported by NSF grant DMS-1803766.}

\author{Jeremy Miller}
\email{jeremykmiller@purdue.edu}
\address{Department of Mathematics \\
	150 North University \\
	West Lafayette IN, 47907 \\USA}
\thanks{Jeremy Miller was supported by NSF grant DMS-1709726.}

\date{May 8, 2020}

\AtBeginDocument{%
	\def\MR#1{}
}

\begin{document}
	
\begin{abstract} 
	We describe an $\cE_k$-cell structure on the free $\cE_{k+1}$-algebra on a point, and more generally describe how the May-Milgram filtration of $\Omega^m \Sigma^m S^{k}$ lifts to a filtration of the free $\cE_{k+m}$-algebra on a point by iterated pushouts of free $\cE_k$-algebras.
\end{abstract}

\maketitle

\section{Introduction} Cell structures on topological spaces have many uses, and operadic cell structures play similar roles in the study of algebras over an operad; they appear when defining model category structures on categories of algebras over operads, and in the case of the little $k$-cubes operad $\cE_k$ they have found applications to homological stability \cite{KM4,GKRW1,GKRW2}. 

Despite the usefulness of such $\cE_k$-cell structures, few have been described explicitly. In this paper we give an $\cE_k$-cell decomposition of an $\cE_k$-algebra weakly equivalent to the free $\cE_{k+1}$-algebra on a point. We generalize this to a filtration of the free $\cE_{k+m}$-algebra on a point by $\cE_k$-algebras, and explain how this can be thought of as a lift of the May-Milgram filtration of the iterated based loop space $\Omega^m \Sigma^m S^{k}$.

We shall state our results without giving definitions (which appear in Section \ref{sec:homotopy-theory-recollection}), with the exception of that of a cell attachment in the category $\cat{Alg}_{\cE_k}(\cat{Top})$. This definition uses the free $\cE_k$-algebra functor $\bfE_k = \bfF^{\cE_k}$, which is left adjoint to the forgetful functor $\cU^{\cE_k} \colon \cat{Alg}_{\cE_k}(\cat{Top}) \to \cat{Top}$ sending an $\cE_k$-algebra to its underlying space. Let $\mathbf{A}$ be an $\cE_k$-algebra in $\cat{Top}$, and let $e \colon \partial D^{d} \to \cU^{\cE_k}(\mathbf{A})$ be a map of topological spaces. To attach a $d$-dimensional $\cE_k$-cell to $\bfA$, we take the adjoint map $\bfE_k(\partial D^d) \to \bfA$ of $e$ and take the pushout of the following diagram in $\cat{Alg}_{\cE_k}(\cat{Top})$:
\[\begin{tikzcd} \bfE_k(\partial D^d) \ar[r] \ar[d] & \bfA \dar  \\
	   \bfE_k(D^d) \rar & \bfA \cup^{\cE_k}_e D^d. \end{tikzcd}\] 
To make this homotopy invariant we need to require that $\bfA$ is cofibrant, or derive the cell-attachment construction.

An $\cE_k$-algebra is called \emph{cellular} if it is built by iterated $\cE_k$-cell attachments (such algebras are always cofibrant). Giving an $\cE_k$-cell structure on $\bfA$ is giving a weak equivalence between $\bfA$ and a cellular $\cE_k$-algebra. Note that the following colimit is also a homotopy colimit: 

\begin{theorem} \label{thm:maincell}
	$\bfE_{k+1}(\ast)$ is weakly equivalent as an $\cE_k$-algebra to a cellular $\cE_k$-algebra with exactly one cell in dimensions divisible by $k$ and no other cells. That is, it is weakly equivalent to the colimit $\mr{colim}_{r \in \bN}\, \bfA_r$ of algebras $\bfA_r$ obtained by setting $\bfA_{-1} = \varnothing$ and taking iterated pushouts in $\cat{Alg}_{\cE_k}(\cat{Top})$
	\[\begin{tikzcd} \bfE_k(\partial D^{rk}) \rar \dar & \bfA_{r} \dar \\
	\bfE_k(D^{rk}) \rar & \bfA_{r+1}. \end{tikzcd}\]
\end{theorem}

An $\cE_k$-cell structure on $\bfA$ induces an ordinary cell structure on its $k$-fold delooping $B^k \bfA$ (see Remark \ref{deloopingCells}). The one induced on $B^k \bfE_{k+1}(\ast) \cong \Omega \Sigma S^k$ by the $\cE_k$-cell structure of Theorem \ref{thm:maincell} is that coming from the James construction \cite{james}. The filtration coming from the James construction was generalized by May and Milgram \cite{M,Mil2} to a filtration on $\Omega^m \Sigma^m S^{k}$, and we will construct a filtration of $\bfE_{k+m}(\ast)$ by $\cE_k$-algebras which deloops to the May-Milgram filtration on $B^k \bfE_{k+m}(\ast) \simeq \Omega^m \Sigma^m S^{k}$.

To state this result precisely, we need to introduce some notation. Let $\bI$ denote the open interval $(0,1)$, $F_r(\bI^m)$ the space of \emph{ordered configurations} of $r$ points in $\bI^m$, and $C_r(\bI^m) \coloneqq F_r(\bI^m)/\fS_r$ the space of \emph{unordered configurations} of $r$ points in $\bI^m$. Furthermore, let $\phi_{m,r}$ denote the vector bundle $F_r(\I^{m}) \times_{\fS_r} \bR^{r-1} \to C_r(\I^m)$ with $\bR^{r-1}$ the representation of the symmetric group $\fS_r$ given by the orthogonal complement to the trivial representation in the permutation representation with its usual metric. This vector bundle inherits a Riemannian metric. For $E \to B$ a vector bundle with Riemannian metric, let $D(E)$ denote its unit disk bundle, $S(E)$ denote its unit sphere bundle and $kE$ denote its $k$-fold Whitney sum.

\begin{theorem} \label{thm:mainhighercell}
	$\bfE_{k+m}(\ast)$ is weakly equivalent as an $\cE_k$-algebra to the colimit $\mr{colim}_{r \in \bN} \, \bfA_r$ of algebras $\bfA_r$ obtained by setting $\bfA_{-1} = \varnothing$ and taking iterated pushouts in $\cat{Alg}_{\cE_k}(\cat{Top})$
	\[\begin{tikzcd} \bfE_k(S(k\phi_{m,r+1})) \rar \dar & \bfA_{r} \dar \\
	\bfE_k(D(k\phi_{m,r+1})) \rar & \bfA_{r+1}. \end{tikzcd}\]
\end{theorem}

This implies that the homotopy cofiber of $\bfA_r \to \bfA_{r+1}$ in $\cat{Alg}_{\cE_k}(\cat{Top})$ is the free $\cE_k$-algebra on the Thom space of $k\phi_{m,r+1}$ viewed as a based space. These Thom spaces and their corresponding Thom spectra are well-studied, e.g.~being related to Brown-Gitler spectra when $k=2$ \cite{cohenmahowaldmilgram,cohencohenkuhnneisendorfer}. When $m=1$, $C_r(\I^m)$ is contractible and the vector bundle $k\phi_{m,r+1}$ has dimension $kr$. Thus the sphere bundle is homotopy equivalent to $\partial D^{rk}$ and hence Theorem \ref{thm:maincell} is a consequence of Theorem \ref{thm:mainhighercell}. 

In Corollary \ref{cor:a-desc} we give a configuration space model $\bfF^{[r]}(\bI^k \times \bI^m)(\ast)$ for the $\cE_k$-algebras $\bfA_r$. The space $\bfF^{[r]}(\bI^k \times \bI^m)(\ast)$ is given by the spaces of configuration spaces of points in $\bI^m \times \bI^k$ such that each subset $\{x\} \times \bI^k$ contains at most $r$ points. We use this in Theorem \ref{mainfilt} to prove that the $k$-fold delooping of $\bfA_r$ is homotopy equivalent to the $r$th stage of the May-Milgram filtration of  $\Omega^m \Sigma^m S^{k}$.

\begin{remark}Our results bear a resemblance to the Dunn--Lurie additivity theorem \cite{Dunn} \cite[Theorem 5.1.2.2]{lurieha}. This result says that $\cE_{k+m} \simeq \cE_k \otimes \cE_m$ for a suitable tensor product of operads, and our result says that $\bfE_{k+m}(\ast)$ can be obtained as $\cE_k$-algebra from the cardinality filtration on $\bfE_m(\ast)$, twisted by the vector bundles $k\phi_{m,r+1}$. It would be interesting to know whether it is possible to deduce Theorem \ref{thm:mainhighercell} from the additivity theorem.
\end{remark}

\subsection*{Acknowledgments} The authors would like to thank Ralph Cohen, S{\o}ren Galatius, David Gepner, Mike Mandell, and Oscar Randal-Williams for helpful conversations. We would also like to thank the anonymous referee for helpful comments and suggestions.
	
\tableofcontents

\section{Recollections of homotopy theory for algebras over an operad}\label{sec:homotopy-theory-recollection}

We work in the setting of \cite{GKRW1} and use similar notation when possible. 

\begin{assumption} $\cat{S}$ is a simplicially enriched complete and cocomplete category with closed symmetric monoidal structure such that the tensor product $\otimes$ commutes with sifted colimits.\end{assumption}

\begin{assumption}$\cat{S}$ comes equipped with a cofibrantly generated model structure which is both simplicial and monoidal, and such that the monoidal unit $\bunit$ is cofibrant. \end{assumption} 

The first version of \cite{GKRW1} required that homotopy equivalences are weak equivalences but this is in fact always the case by Proposition 9.5.16 of \cite{hirschhorn}.

When $\cat{G}$ is a symmetric monoidal category, then we may endow the category $\cat{S}^\cat{G}$ of functors $\cat{G} \to \cat{S}$ with the Day convolution tensor product; this will also be symmetric monoidal. Similarly, the category $\cat{S}_\ast$ of pointed objects in $\cat{S}$ with smash product inherits these properties.

\begin{example}\label{exam:cats} The examples of $\cat{S}$ most relevant to this paper are: (i) the category  $\cat{sSet}$ of simplicial sets with the Quillen model structure and cartesian product, and (ii) the category $\cat{Top}$ of compactly generated weakly Hausdorff spaces with the Quillen model structure and cartesian product (see \cite{stricklandcgwh} for more details about the point-set topology).
\end{example}

Let $\cat{FB}_\infty$ denote the category of (possibly empty) finite sets and bijections, then the objects of the category  $(\cat{S}^\cat{G})^{\cat{FB}_\infty}$ of functors $\cat{FB}_\infty \to \cat{S}^\cat{G}$ are called \emph{symmetric sequences}. In addition to the Day convolution tensor product, $(\cat{S}^\cat{G})^{\cat{FB}_\infty}$ admits a composition product $\circ$ (which is rarely symmetric); for $\cX,\cY \in (\cat{S}^\cat{G})^{\cat{FB}_\infty}$ the evaluation of the composition product $\cX \circ \cY$ on the set $\{1,2,\ldots,r\}$ is given by
\[\cX \circ \cY(r) = \bigsqcup_{d \geq 0} \cX(d) \otimes_{\fS_d} \left(\bigsqcup_{k_1+\ldots+k_d = r}  \fS_r  \times_{\fS_{k_1}+\ldots+\fS_{k_d}}\cY(k_1) \otimes \ldots \otimes \cY(k_d)\right).\]

A (symmetric) \emph{operad} is a unital monoid with respect to this composition product. An $\cO$-algebra $\bfA$ is an object $A \in \cat{S}$ with a left $\cO$-module structure on $A$ considered as a symmetric sequence concentrated in cardinality $0$. Equivalently we can use the associated monad  on $\cat{S}$, for which we also use the notation $\cO$,
\[\cO(X) \coloneqq \bigsqcup_{r \geq 0} \cO(r) \otimes_{\fS_{r}} X^{\otimes r},\]
and define an $\cO$-algebra to be an algebra over this monad. The category $\cat{Alg}_\cO(\cat{S}^\cat{G})$ of $\cO$-algebras is both complete and cocomplete.

A \emph{free} $\cO$-algebra is one of the form $\cO(X)$ with $\cO$-algebra structure maps induced by the monad multiplication and unit. We use the notation $\bfF^\cO \colon \cat{S}^\cat{G} \to \cat{Alg}_\cO(\cat{S}^\cat{G})$ for the free $\cO$-algebra functor, which is the left adjoint to the forgetful functor $\cU^\cO \colon \cat{Alg}_\cO(\cat{S}^\cat{G}) \to \cat{S}^\cat{G}$ sending an algebra to its the underlying object. Note $\cO = \cU^\cO \bfF^\cO$.

Any $\cO$-algebra admits a canonical presentation as a reflexive coequalizer of free $\cO$-algebras:
\[\begin{tikzcd} \bfF^\cO(\cO(\cU^\cO(\bfA))) \arrow[shift left=.5ex]{r} \arrow[shift left=-.5ex]{r} & \bfF^\cO(\cU^\cO(\bfA)) \rar & \bfA,\end{tikzcd}\]
the top map coming from the $\cO$-algebra structure map $\cO(\cU^\cO(\bfA)) \to \cU^\cO(\bfA)$, and the bottom map coming from the natural transformation $\bfF^\cO \cO \to \bfF^\cO$ induced by the monad multiplication. Thus free $\cO$-algebras generate the category of $\cO$-algebras under sifted colimits, and the category of right $\cO$-module functors $\cat{C} \to \cat{D}$ preserving sifted colimits is equivalent to the category of functors $\cat{Alg}_\cO(\cat{C}) \to \cat{D}$ preserving sifted colimits; one constructs the latter from the former using the canonical presentation, and one constructs the former from the latter by evaluating on free $\cO$-algebras.

For $X \in \cat{S}^\cat{G}$ and $g \in \cat{G}$, the evaluation $X \mapsto X(g) \in \cat{S}$ has a left adjoint; given $Y \in \cat{S}$ we denote its image under this left adjoint by $Y^g$. Given a map $\partial D^d \to \cU^\cO(\bfA)(g)$ (where $\partial D^d$ stands for $\partial D^d \otimes \bunit$, the copowering of $\partial D^d$ with the monoidal unit), we obtain by adjunction first a map $\partial D^{g,d} \to \cU^\cO(\bfA)$ and then a map $\bfF^\cO(\partial D^{g,d}) \to \bfA$. An \emph{$\cO$-cell attachment} is defined to be the following pushout in $\cat{Alg}_\cO(\cat{S}^\cat{G})$
\begin{equation} \label{eqn:cell-attachment} \begin{tikzcd}\bfF^\cO(\partial D^{g,d}) \dar \rar & \bfA \dar \\
\bfF^\cO(D^{g,d}) \rar & \bfA \cup^\cO D^{g,d}.\end{tikzcd}\end{equation}
Explicitly this pushout may be constructed as the following reflexive coequalizer
\[\begin{tikzcd} \bfF^\cO(\cO(\cU^\cO(\bfA)) \cup D^{g,d}) \arrow[shift left=.5ex]{r} \arrow[shift left=-.5ex]{r} & \bfF^\cO(\cU^\cO(\bfA) \cup D^{g,d}) \rar & \bfA \cup^\cO D^{g,d}.\end{tikzcd}\]
The left vertical map in (\ref{eqn:cell-attachment}) is a cofibration, so cell attachments are homotopy-invariant when $\cat{Alg}_\cO(\cat{S}^\cat{G})$ is left proper. In general we need to derive the construction; we will momentarily explain when this can be done using a monadic bar resolution.

Using the copowering of $\cat{S}^\cat{G}$ over $\cat{sSet}$, any operad in simplicial sets gives rise to an operad in $\cat{S}^\cat{G}$, and using the strong monoidal functor $\mr{Sing} \colon \cat{Top} \to \cat{sSet}$ so does any operad in compactly-generated weakly Hausdorff topological spaces. We shall restrict our attention to operads $\cO$ in $\cat{sSet}$ which are $\Sigma$-cofibrant, i.e.\ for all $r \geq 0$ the $\fS_r$-action on $\cO(r)$ is free. We may attempt to define a model structure on $\cat{Alg}_\cO(\cat{S}^\cat{G})$ by declaring the (trivial) fibrations and weak equivalences to be those of underlying objects. If it exists, this is called the projective model structure.

\begin{assumption}The projective model structure exists on $\cat{Alg}_\cO(\cat{S}^\cat{G})$.\end{assumption}

When this assumption is satisfied, the projective model structure will be a cofibrantly generated model structure with generating (trivial) cofibrations obtained by applying $\cO$ to the generating (trivial) cofibrations of the model structure on $\cat{S}^\cat{G}$. Since $\cat{S}^\cat{G}$ is a simplicial and monoidal model category, it is automatic that the forgetful functor $U^\cO\colon \cat{Alg}_\cO(\cat{S}^\cat{G}) \to \cat{S}^\cat{G}$ preserves (trivial) cofibrations, cf.\ Lemma 9.5 of \cite{GKRW1}. When $\cO$ is a $\Sigma$-cofibrant operad in simplicial sets, the projective model structure exists in the settings of Example \ref{exam:cats}, cf.\ Section 9.2 of \cite{GKRW1}.

When $\cU^\cO(\bfA) \in \cat{S}^\cat{G}$ is cofibrant, we may use the monadic bar resolution to find an explicit cofibrant replacement of $\bfA$ and thus compute derived functors.
 
\begin{definition}The \emph{monadic bar resolution} is the augmented simplicial object $B_\bullet(\bfF^\cO,\cO,\bfA)$ with $p$-simplices given by $\bfF^\cO(\cO^p(\cU^\cO(\bfA)))$ for $p \geq 0$ and $\bfA$ for $p=-1$. The face maps and augmentation are induced by the monad multiplication and the $\cO$-algebra structure on $\bfA$, and the degeneracies by the unit of the monad.\end{definition}

This is a special case of the \emph{two-sided monadic bar construction}, which is used throughout the paper. It takes as input a monad $T$, a right $T$-functor $F$ and a $T$-algebra $\bfA$ with underlying object $A$, and has $p$-simplices given by $B_\bullet(F,T,\bfA) = F(T^p(A))$. The face maps and degeneracy maps are similar to above, for details see e.g.\ Section 9 of \cite{M}.

Let $|-|$ denote the (thin) geometric realization, and introduce the notation $B(\bfF^\cO,\cO,\bfA) \coloneqq |B_\bullet(\bfF^\cO,\cO,\bfA)|$. Note that here we take geometric realization in the categories of $\cO$-algebras, but $\cU^\cO$ commutes with geometric realization by Section 8.3.3 of \cite{GKRW1}. The augmentation induces a map $B(\bfF^\cO,\cO,\bfA) \to \bfA$, which is always a weak equivalence using an extra degeneracy argument. It is a free simplicial resolution in the sense of Definition 8.16 of \cite{GKRW1} when the bar construction is Reedy cofibrant. Because $\cO$ is $\Sigma$-cofibrant, this is the case when $\cU^\cO(\bfA)$ is cofibrant, using the Reedy cofibrancy criterion of Lemma 9.14 of \cite{GKRW1}.

\begin{remark}\label{deloopingCells} In \cite{KM4}, $\cO$-algebra cell attachments were defined using partial algebras. The formula in Definition 3.1 of \cite{KM4}, written in our notation, is $|[p] \mapsto \bfF^\cO(\cO^p(\bfA) \cup D^d)|$. This may be obtained by inserting $B(\bfF^\cO,\cO,\bfA)$ into the underived formula for cell attachment. We explained above that this gives derived cell attachment when $\cU^\cO(\bfA)$ is cofibrant, but in $\cat{Top}$ and $\cat{sSet}$ this assumption is unnecessary. In $\cat{sSet}$, every object is cofibrant. In $\cat{Top}$, we use that geometric realization sends levelwise weak equivalences between proper simplicial spaces to weak equivalences, even if the simplicial spaces are not levelwise cofibrant. This allows us to cofibrantly replace $\bfA$ in the category of $\cO$-algebras (which will be cofibrant in topological spaces because $\cO$ is $\Sigma$-cofibrant). Thus results of \cite{KM4} apply: in particular, Proposition 6.12 of \cite{KM4} implies that an $\cE_k$-cell structure on an $\cE_k$-algebra in topological spaces deloops to an ordinary cell structure on the $k$-fold delooping. The reason for this is that delooping preserves homotopy pushouts and $\bfE_k (\partial D^d) \to \bfE_k( D^d) $ deloops to $\partial D^d \hookrightarrow  D^d$.

\end{remark}

\section{Rank completion} We shall define a rank completion filtration in the case that we are working in a category of functors $\cat{S}^\cat{G}$ where $\cat{G}$ has a notion of rank, and the operad $\cO$ and $\cO$-algebra $\bfA$ satisfy mild conditions. Later in this paper, $\cO$ will be $\cE_k$ and $\cat{G}$ will be $\bN$; the rank function will be used to keep track of the number of points in a configuration, and though we shall not use this, $\cat{G}$ can be used to record group actions on configurations. 

\begin{assumption}$\cat{G}$ is a symmetric monoidal groupoid equipped with strong monoidal functor $\kappa \colon \cat{G} \to \bN$, which we call a \emph{rank functor}.
\end{assumption} 

Let $\cat{G}_{\leq r}$ denote the full subcategory on $\cat{G}$ on those objects $g$ such that $\kappa(g) \leq r$, and $\cat{G}_r$ denote the full subcategory on objects $g$ such that $\kappa(g) = r$. Precomposition gives restriction functors $(\leq r)^*$ and $(r)^*$ participating in adjunctions
\[\begin{tikzcd} \cat{S}^{\cat{G}_{\leq r}} \arrow[shift left=.5ex]{r}{(\leq r)_*} & \cat{S}^{\cat{G}} \arrow[shift left=.5ex]{l}{(\leq r)^*}, \end{tikzcd} \qquad \begin{tikzcd} \cat{S}^{\cat{G}_{r}} \arrow[shift left=.5ex]{r}{(r)_*} & \cat{S}^{\cat{G}} \arrow[shift left=.5ex]{l}{(r)^*}. \end{tikzcd}\]
There are further relative restriction and extension functors between $\cat{S}^{\cat{G}_r}$, $\cat{S}^{\cat{G}_{\leq r}}$ for different $r$, participating in analogous adjunctions. The functors $(\leq r)^*$ and $(r)^*$ are themselves left adjoints; though we will not use their right adjoints, we will use that $(\leq r)^*$ and $(r)^*$ commute with colimits.

It follows from the formula for Day convolution that $\cat{S}^{\cat{G}_{\leq r}}$ inherits a symmetric monoidal tensor product, an alternative expression for which is given by $X \otimes Y = (\leq r)^* ((\leq r)_*(X) \otimes (\leq r)_*(Y))$. This makes visible that $(\leq r)^*$ is strong monoidal and simplicial. In particular, the functor $(\leq r)^*$ takes $\cO$-algebras in $\cat{S}^\cat{G}$ to $\cO$-algebras in $\cat{S}^{\cat{G}_{\leq r}}$. Its left adjoint $(\leq r)_*$ in general does not. However, we may use the canonical presentation of $\cO$-algebras explained in the previous section to construct a left adjoint $(\leq r)_*^\mr{alg} \colon \cat{Alg}_\cO(\cat{S}^{\cat{G}_{\leq r}}) \to \cat{Alg}_\cO(\cat{S}^\cat{G})$ to $(\leq r)^* \colon \cat{Alg}_\cO(\cat{S}^\cat{G}) \to \cat{Alg}_\cO(\cat{S}^{\cat{G}_{\leq r}})$. Explicitly it is the following reflexive coequalizer
\[\begin{tikzcd} \bfF^\cO((\leq r)_*\cO(\cU^\cO(\bfA))) \arrow[shift left=.5ex]{r} \arrow[shift left=-.5ex]{r} & \bfF^\cO((\leq r)_* \cU^\cO(\bfA)) \rar & (\leq r)_*^\mr{alg}(\bfA).\end{tikzcd}\]
It is defined uniquely up to isomorphism by demanding that $(\leq r)_*^\mr{alg}  F^\cO(X) = F^\cO((\leq r)_*(X))$ and that it preserves sifted colimits.

\begin{definition}We define the \emph{$r$th rank completion functor} $\bfT_r \colon \cat{Alg}_\cO(\cat{S}^\cat{G}) \to \cat{Alg}_\cO(\cat{S}^\cat{G})$ to be $(\leq r)_*^\mr{alg} (\leq r)^*$.\end{definition}

This functor underlies the monad associated to the adjunction $(\leq r)_*^\mr{alg} \dashv (\leq r)^*$ and has a right adjoint. The counit gives a natural transformation $\bfT_r \Rightarrow \mr{id}$, and the commutative diagram of groupoids
\[\begin{tikzcd} \cat{G} & & & & \\
\cat{G}_{\leq 0} \rar \uar & \cat{G}_{\leq 1} \arrow{lu} \rar & \cat{G}_{\leq 2} \rar \arrow{llu} & \cdots, \end{tikzcd}\]
gives rise to a tower of natural transformations of functors $\cat{Alg}_\cO(\cat{S}^\cat{G}) \to \cat{Alg}_\cO(\cat{S}^\cat{G})$
\[\begin{tikzcd} \mr{id} & & & & \\
\bfT_0 \rar \uar & \bfT_1 \rar  \arrow{lu} & \bfT_2 \rar \arrow{llu} & \cdots. \end{tikzcd}\]
Since colimits are computed objectwise and the map $(g)^* \bfT_r(\bfA) \to (g)^* \bfA$ is the identity as soon as $r \geq \kappa(g)$, the natural transformation $\mr{colim}_{r \in \bN} \bfT_r \Rightarrow \mr{id}$ is a natural isomorphism.

The functor $(\leq r)^*$ obviously preserves fibrations and weak equivalences, so $(\leq r)_*^\mr{alg}$ is a left Quillen functor. However, $(\leq r)^*$ also preserves cofibrations, as these are retracts of iterated pushouts along free $\cO$-algebra maps, which are preserved by $(\leq r)^*$. Hence $\bfT_r = (\leq r)_*^\mr{alg} (\leq r)^*$ preserves trivial cofibrations between cofibrant objects, and thus admits a left derived functor by precomposition with a functorial cofibrant replacement. Moreover, as explained above, when $\cU^\cO(\bfA)$ is cofibrant we may use a monadic bar resolution to cofibrantly replace $\bfA$. As a composition of two left adjoints, $\bfT_r$ commutes with geometric realization. Thus we obtain the following formula for $\bfT_r^\bL(\bfA)$:
\begin{equation}\label{eqn:tr-explicit} \bfT_r^\bL(\bfA) = B(\bfF^\cO (\leq r)_*,\cO,(\leq r)^* \cU^\cO(\bfA)) =  B(\bfF^\cO (\leq r)_* (\leq r)^*,\cO, \cU^\cO(\bfA)),\end{equation}
the latter equality following from the fact that $(\leq r)^*$ commutes with $\cO$.

We next restrict our attention to a setting where the underlying object of $\bfF^\cO(X)$ agrees with $X$ in rank $\leq r$ up to homotopy, for those $X$ which are concentrated in rank $r$. To see when this occurs, note that for any operad $\cO$ and $X$ concentrated in rank $r$, $\cO(X)$ is isomorphic to $\cO(0)$ in rank $0$ and $\cO(1) \otimes X$ in rank $r$. Hence the following assumption suffices:

\begin{assumption}\label{ass:unitary} $\cO$ is a \emph{non-unitary} operad in simplicial sets, i.e.\ $\cO(0) = \varnothing$, and $\cO(1) \simeq \ast$.
\end{assumption}

\begin{definition}We say $X \in \cat{S}^\cat{G}$ is \emph{reduced} if it is concentrated in rank $>0$, that is, $X(g)$ is initial when $\kappa(g) = 0$.\end{definition}

The horizontal maps in the following proposition are obtained from the identity morphisms of $(r+1)^* \cU^\cO \bfT_r^\bL(\bfA)$ and $(r+1)^* \cU^\cO \bfT_{r+1}^\bL(\bfA)$ respectively, the vertical maps from the natural transformation $\bfT_r \Rightarrow \bfT_{r+1}$.

\begin{proposition}\label{prop:attach-one-step} For reduced $\bfA$ there is a homotopy cocartesian square in $\cat{Alg}_\cO(\cat{S}^{\cat{G}})$
	\[\begin{tikzcd}\bfF^{\cO}((r+1)_* (r+1)^* \cU^\cO \bfT^\bL_r(\bfA)) \dar \rar &   \bfT^\bL_r(\bfA) \dar \\
	\bfF^{\cO}((r+1)_* (r+1)^* \cU^\cO \bfT^\bL_{r+1}(\bfA))) \rar & \bfT^\bL_{r+1}(\bfA), \end{tikzcd} \]
where we remark that $(r+1)^* \cU^\cO \bfT^\bL_{r+1}(\bfA) \cong (r+1)^* \cU^\cO(\bfA)$.
\end{proposition}

\begin{proof}This diagram is obtained by applying $(\leq r+1)_*^\mr{alg}$ to a diagram in $\cat{Alg}_\cO(\cat{S}^\cat{G}_{\leq r+1})$, a functor which preserves homotopy cocartesian squares as it is a left Quillen functor. Hence it suffices to prove that the following is homotopy cocartesian in $\cat{Alg}_\cO(\cat{S}^{\cat{G}_{\leq r+1}})$:
	\[\begin{tikzcd}\bfF^{\cO}((r+1)^* \cU^\cO \bfT^\bL_r(\bfA)) \dar \rar &  (\leq r)^\mr{alg}_* (\leq r)^* \bfA \dar \\
\bfF^{\cO}((r+1)^* \cU^\cO \bfA) \rar & (\leq r+1)^* \bfA, \end{tikzcd} \]
	where $\bfF^\cO$ now denotes the free algebra functor $\cat{S}^{\cat{G}_{\leq r+1}} \to \cat{Alg}_\cO(\cat{S}^{\cat{G}_{\leq r+1}})$, $(\leq r)_*$ and $(r+1)_*$ denote the left adjoints to $(\leq r)^* \colon \cat{S}^{\cat{G}_{\leq r+1}} \to \cat{S}^{\cat{G}_{\leq r}}$ and $(\leq r+1)^* \cat{S}^{\cat{G}_{\leq r+1}} \to \cat{S}^{\cat{G}_{r+1}}$ respectively, and $(\leq r)^\mr{alg}_*$ denotes the left adjoint to $(\leq r)^* \colon \cat{Alg}_\cO(\cat{S}^{\cat{G}_{\leq r+1}}) \to \cat{Alg}_\cO(\cat{S}^{\cat{G}_{\leq r}})$.

The result now follows from the next lemma: substitute in its statement
\begin{align*}X &\leadsto  (r+1)^* \cU^\cO \bfT^\bL_r(\bfA) \\
Y &\leadsto (r+1)^* \cU^\cO \bfA \\
\bfA &\leadsto (\leq r)^\mr{alg}_* (\leq r)^* \bfA \\
\bfB &\leadsto \bfA.\end{align*}
Verifying condition (ii) uses that $\cO(1) \simeq \ast$.
\end{proof}

\begin{lemma} Suppose $\bfA,\bfB \in\cat{Alg}_\cO(\cat{S}^{\cat{G}_{\leq r+1}})$ are cofibrant in $\cat{S}^{\cat{G}_{\leq r+1}}$ and reduced, and $X,Y \in \cat{S}^{\cat{G}_{\leq r+1}}$ are cofibrant and concentrated in rank $r+1$. Then a commutative square
	\[\begin{tikzcd}\bfF^{\cO}(X) \dar \rar &  \bfA \dar \\
	\bfF^{\cO}(Y) \rar & \bfB, \end{tikzcd}\]
is homotopy cartesian in $\cat{Alg}_\cO(\cat{S}^{\cat{G}_{\leq r+1}})$ if the following two conditions hold:
\begin{enumerate}[\indent (i)]
	\item the map $(\leq r)^* \bfA \to (\leq r)^* \bfB$ is a weak equivalence,
	\item the commutative square
		\[\begin{tikzcd}(r+1)^* X \dar \rar &  (r+1)^*\cU^\cO\bfA \dar \\
	(r+1)^* Y \rar & (r+1)^*\cU^\cO\bfB \end{tikzcd}\]
	is homotopy cocartesian.
\end{enumerate}
\end{lemma}

\begin{proof}We may assume without loss of generality that $\bfA$ and $\bfB$ are  cofibrant in $\cat{S}^{\cat{G}_{\leq r+1}}$ and $X \to Y$ is a cofibration between cofibrant objects. We can factor the commutative square as
	\[\begin{tikzcd} \bfF^{\cO}(X) \rar \dar & B(\bfF^\cO,\cO,\cU^\cO(\bfA)) \rar{\simeq} \dar & \bfA \dar \\
	\bfF^{\cO}(Y) \rar & B(\bfF^\cO,\cO,\cU^\cO(\bfB)) \rar{\simeq} & \bfB,\end{tikzcd}\]
where the horizontal maps are weak equivalences because $\bfA$ and $\bfB$ are cofibrant in $\cat{S}^{\cat{G}_{\leq r+1}}$.

The left square is the geometric realization of the following square of simplicial objects
	\[\begin{tikzcd} \left([p] \mapsto \bfF^{\cO}(X)\right) \rar \dar & \left([p] \mapsto \bfF^\cO(\cO^p(\cU^\cO(\bfA)))\right) \dar  \\
\left([p] \mapsto \bfF^{\cO}(Y)\right) \rar & \left([p] \mapsto \bfF^\cO(\cO^p(\cU^\cO(\bfB)))\right).\end{tikzcd}\]
All these simplicial objects are Reedy cofibrant; this is evident for the left entries, and for the right entries follows from another application of Lemma 9.14 of \cite{GKRW1}. Geometric realization of Reedy cofibrant simplicial objects is a homotopy colimit, and thus commutes with homotopy pushouts. In particular, a levelwise homotopy cocartesian diagram of Reedy cofibrant simplicial objects geometrically realizes to a homotopy cocartesian diagram. It thus suffices to prove that each of the levels
	\[\begin{tikzcd}\bfF^{\cO}(X) \rar \dar & \bfF^\cO(\cO^p(\cU^\cO(\bfA))) \dar  \\
	\bfF^{\cO}(Y) \rar & \bfF^\cO(\cO^p(\cU^\cO(\bfB)))\end{tikzcd}\]
is homotopy cocartesian. 

Since $X \to Y$ is a cofibration between cofibrant objects, the map from the homotopy pushout to the bottom-right corner is given by
\[\bfF^\cO(\cO^p(\cU^\cO(\bfA)) \cup_X Y) \lra \bfF^\cO(\cO^p(\cU^\cO(\bfB))).\]
This is a weak equivalence in $\cat{Alg}_{\cO}(\cat{S}^{\cat{G}_{\leq r+1}})$ if and only if the map on underlying objects is. Since $\bfF^\cO$ preserves weak equivalences between cofibrant objects, it suffices to prove that the map
\[\cO^p(\cU^\cO(\bfA)) \cup_X Y \lra \cO^p(\cU^\cO(\bfB))\]
is a weak equivalence. Indeed, both objects are cofibrant since $X \mapsto \cO(X)$ preserves cofibrant objects, as do pushouts along a cofibration.

We do this by induction over $p$. For $p=0$, we observe that since $X$ and $Y$ are concentrated in rank $r+1$, we have a commutative diagram 
\[\begin{tikzcd}(\leq r)^* \cU^\cO(\bfA) \rar \dar{\cong} &  (\leq r)^* \cU^\cO(\bfB) \dar[equals] \\
(\leq r)^* \cU^\cO(\bfA) \cup_X Y \rar & (\leq r)^* \cU^\cO(\bfB).\end{tikzcd}\]
Thus for ranks $\leq r$ the result follows from assumption (i). In rank $r+1$, the case $p=0$ follows from assumption (ii). This completes the proof of the initial case.

To prove the induction step, it suffices to prove the following statement: if $Z,W$ are reduced and $X,Y$ are concentrated in degree $r+1$, then if (i) $(\leq r)^* Z \to (\leq r)^* W$ is a weak equivalence and (ii) the commutative square
\[\begin{tikzcd}(r+1)^* X \dar \rar &  (r+1)^* Z \dar \\
(r+1)^* Y \rar & (r+1)^* W \end{tikzcd}\]
is homotopy cocartesian, then (i') $(\leq r)^* \cO(Z) \to (\leq r)^* \cO(W)$ is a weak equivalence and (ii') the commutative square
\begin{equation}\label{eqn:diag-final-po} \begin{tikzcd}(r+1)^* X \dar \rar &  (r+1)^*\cO(Z) \dar \\
(r+1)^* Y \rar & (r+1)^*\cO(W)\end{tikzcd}\end{equation}
is homotopy cocartesian.

Deducing (i) from (i') is done by noting that $(\leq r)^*$ commutes with $\cO$ and $\cO$ preserves weak equivalences between cofibrant objects. To deduce (ii') from (i) and (ii), we use the formula
\[(r+1)^* \cO(Z) = \bigsqcup_{n \geq 1} \cO(n) \otimes_{\fS_{n}} \left(\bigsqcup_{\substack{1 \leq r_1,\ldots,r_n \leq r+1 \\ \sum r_i = r+1}} (r_1)_*(r_1)^* Z \otimes \cdots \otimes (r_n)_* (r_n)^* Z\right)\]
and a similar one for $W$. To restrict the $r_i$ to positive integers, we used that $\cO$ is non-unitary, and that $Z$ and $W$ are reduced. From this expression, we see that \eqref{eqn:diag-final-po} is a coproduct of two commutative diagrams. The first is
\[\begin{tikzcd}\mathbbm{i}\dar \rar &  (r+1)^*\cO((\leq r)_* (\leq r)^*Z) \dar \\
\mathbbm{i}\rar & (r+1)^*\cO((\leq r)_* (\leq r)^* W), \end{tikzcd}\]
where $\mathbbm{i}$ is the initial object, which is homotopy cocartesian because the right map is a weak equivalence as a consequence of (i). The second is
\[\begin{tikzcd}(r+1)^* X \dar \rar &  \cO(1) \otimes (r+1)^* Z\dar \\
(r+1)^* Y\rar &  \cO(1) \otimes (r+1)^* W,\end{tikzcd}\]
which homotopy cocartesian by (ii) since $\cO(1) \simeq \ast$.
\end{proof}

We thus get a sequence of maps
\[\bfT_0^\bL(\bfA) \lra \bfT_1^\bL(\bfA) \lra \bfT_2^\bL(\bfA) \lra \cdots\]
whose homotopy colimit is naturally weakly equivalent to $\bfA$ and whose homotopy cofibers we understand. This is the  \emph{rank completion filtration}. 

When we can make sense of homology, e.g.\ in one of the settings mentioned in Section 10.1 of \cite{GKRW1}, we get a corresponding spectral sequence converging conditionally to the homology of $\cU^\cO(\bfA)$. The $E^1$-page will be rather unwieldy, and we believe the following spectral sequence may be more useful:

\begin{remark}\label{rem:o-homology-ss} 
Let $(-)_+$ denote the monad whose underlying functor takes the coproduct with the terminal object (so that algebras over it are pointed objects). As $\cO$ is a \emph{non-unitary} operad in simplicial sets, cf.\ Assumption \ref{ass:unitary}, there is a canonical map of monads from $\cO$ to $(-)_+$ which can be viewed as an augmentation of $\cO$. This augmentation is given on $X \in \cat{S}^\cat{G}$ by the map $\cO(X) = \bigsqcup_{n \geq 1} \cO(n) \otimes_{\fS_{n}} X^{\otimes n} \to X_+$ which on the summand $\cO(1) \otimes X$ is the map $\cO(1) \otimes X \to \ast \otimes X = X$ and on the summands $\cO(n) \otimes_{\fS_{n}} X^{\otimes n}$ for $n \geq 2$ is the unique map to the terminal object.

Taking indecomposables with respect to this augmentation, we obtain the $\cO$-indecomposables functor $Q^\cO \colon \cat{Alg}_\cO(\cat{S}^\cat{G}) \to \cat{S}_\ast^\cat{G}$ determined uniquely up to isomorphism by demanding that $Q^\cO F^\cO \cong (-)_+$ and that $Q^\cO$ commutes with sifted colimits. Applying its left-derived functor $Q^\cO_\bL$ to the diagram in the previous proposition, we get that if $\bfA$ is reduced there is a homotopy cocartesian square in $\cat{S}_*^\cat{G}$:
		\[\begin{tikzcd}(r+1)_* (r+1)^* \cU^\cO \bfT^\bL_r(\bfA)_+ \dar \rar &  Q^{\cO}_\bL(\bfT^\bL_r (\bfA)) \dar \\
	(r+1)_* (r+1)^* \cU^\cO \bfT^\bL_{r+1}(\bfA)_+ \rar & Q^\cO_\bL(\bfT^\bL_{r+1}(\bfA)). \end{tikzcd} \]

When we can make sense of homology, we can define $\cO$-homology by $\smash{H^{\cO}_{g,d}}(\bfA) \coloneqq \smash{\tilde{H}_d((g)^*Q^\cO_\bL(\bfA))}$. The result of the previous discussion is a conditionally convergent spectral sequence (suppressing the filtration degree, so in particular the $p$ in $E^1_{p,q}$ refers to rank)
	\[E^1_{p,q} = H_{p+q}((p)^* \bfT^\bL_p(\bfA),(p)^* \bfT^\bL_{p-1}(\bfA)) \Longrightarrow H^\cO_{p,p+q}(\bfA),\]
where it may be helpful to recall that $(p)^* \bfT^\bL_p(\bfA) \cong (p)^* \bfA$. 
\end{remark}

\section{An $\cE_k$-algebraic analogue of the May-Milgram filtration} To deduce our results, we specialize the results of the previous section to $\cO = \cE_k$, the non-unital little $k$-cubes operad. Recall that $\I$ denotes the open interval $(0,1)$, and let $\mr{Emb}^\mr{rect}(\bigsqcup_n \bI^k,\bI^k)$ denote the space of ordered $n$-tuples of rectilinear embeddings $\I^k \to \I^k$ with disjoint image (that is, they are a composition of translation and dilation by positive real numbers in each of the $k$ directions).

\begin{definition}The \emph{non-unital little $k$-cubes operad} $\cE_k$ has topological space $\cE_k(n)$ of $n$-ary operations given by
	\[\cE_k(n) \coloneqq \begin{cases} \varnothing & \text{if $n=0$,} \\
	\mr{Emb}^\mr{rect}(\bigsqcup_n \bI^k,\bI^k) & \text{if $n>0$,}\end{cases}\]
	with symmetric group $\fS_n$ permuting the $n$-tuples. The unit in $\cE_k(1)$ is the identity map $\I^k \to \I^k$, and composition is induced by composition of embeddings.\end{definition}

This satisfies Assumption \ref{ass:unitary} and hence gives rise to an operad in $\cat{S}^\cat{G}$, all of whose objects are concentrated on the monoidal unit of $\cat{G}$. $\cE_k$-algebras in $\cat{S}^\cat{G}$ are algebras over this operad, and we shall adopt the shorter notation $\bfE_k$ for the free $\cE_k$-algebra functor $\bfF^{\cE_k}$. (If $\cat{S} = \cat{Top}$ and $\cat{G} = \ast$, as a consequence of our conventions these are algebras over the operad $|\mr{Sing}(\cE_k)|$ in topological spaces.)

We shall take $\cat{G} = \bN$, with $\kappa \colon \bN \to \bN$ the identity functor. Let \[\cU^{\cE_{k+m}}_{\cE_k} \colon \cat{Alg}_{\cE_{k+m}}(\cat{S}^\bN) \lra \cat{Alg}_{\cE_{k}}(\cat{S}^\bN)\]
denote the forgetful functor induced by the map of operads $\cE_{k} \to \cE_{k+m}$ given by sending a cube $e \colon \I^k \to \I^k$ to $e \times \mr{id}_{\I^m} \colon \I^k \times \I^m \to \I^k \times \I^m$. For the sake of brevity we will often write $\cU$ for $\cU^{\cE_{k+m}}_{\cE_k}$.

We are interested in free algebras on a point, which we will consider concentrated in rank $1$. In this section we will more generally study $\bfE_{k+m}(X)$ for $X \in \cat{S}^\bN$ satisfying a similar condition:

\begin{assumption}$X \in \cat{S}^\bN$ is concentrated in rank $1$, i.e.\ $X(g)$ is initial unless $g=1$ (so in particular reduced), and $X$ is cofibrant.
\end{assumption}

We will give an elementary geometric model for the $\cE_k$-algebra $\bfT^\bL_r(\cU \bfE_{k+m}(X))$, and use Proposition \ref{prop:attach-one-step} to describe $\cU \bfE_{k+m}(X)$ up to weak equivalence as a colimit of iterated pushouts along maps of free $\cE_k$-algebras. The following definition was mentioned in the introduction:

\begin{definition}For a manifold $M$ and $n \geq 1$, the topological space $F_n(M)$ of \emph{ordered configurations of $n$ points in $M$} is given by $\{(m_1,\ldots,m_n) \mid \text{$m_i \neq m_j$ if $i \neq j$}\} \subset M^n$. For $n=0$ we define $F_n(M) = \varnothing$.

\end{definition}

We choose to define $F_0(M)$ to be empty since we work with non-unital $\cE_k$-algebras. For $n>0$, the topological space $F_n(M)$ is homeomorphic to the space of embeddings of the set $\{1,\ldots,n\}$ into $M$, and precomposition by permutations of $\{1,\ldots,n\}$ defines a $\fS_n$-action on the space $F_n(M)$. Taking the singular simplicial set, these assemble to a symmetric sequence $F(M)$ in $\cat{sSet}$. For $M=\I^k \times \I^m$, this is a left $\cE_{k+m}$-module, where the action is given by composition of embeddings. 

Just like we used the enrichment of copowering of $\cat{S}^\bN$ over $\cat{sSet}$ to make the operad $\cE_k$ in $\cat{sSet}$ into an operad in $\cat{S}^\bN$, we use it to make the left $\cE_{k+m}$-module $F(\I^k \times \I^m)$ in $\cat{sSet}$ into a left $\cE_{k+m}$-module in $\cat{S}^\bN$. Analogously to the free $\cE_k$-algebra construction, we can take the composition product of $F(\I^k \times \I^m) \in (\cat{S}^\bN)^\cat{G}$ with an object $X \in \cat{S}^\bN$ considered as a symmetric sequence concentrated in cardinality $0$. We refer to this as ``applying'' $F(\I^k \times \I^m)$ to $X$. The resulting object $F(\I^k \times \I^m)(X) \in \cat{S}^\bN$ comes endowed with an $\cE_{k+m}$-algebra structure. This construction is natural in $X$, and thus we obtain a functor $\bfF(\bI^k \times \bI^m) \colon \cat{S}^\bN \to \cat{Alg}_{\cE_{k+m}}(\cat{S}^\bN)$. 

\begin{definition}We let $F^{[r]}_n(\I^k \times \I^m)$ denote the subspace of $F_n(\I^k \times \I^m)$ consisting of ordered configurations $\eta = (m_1,\ldots,m_n)$ such that for all $x \in \bI^k$ the intersection $\eta \cap (\{x\} \times \bI^m)$ has cardinality at most $r$.\end{definition}

\begin{figure}
	\begin{tikzpicture}
		\draw (0,0) rectangle (4,4);
		\node at (0.5,1) {$\bullet$};
		\node at (0.5,1.5) {$\bullet$};
		\node at (1.5,.3) {$\bullet$};
		\node at (1.5,1.5) {$\bullet$};	
		\node at (1.5,1.9) {$\bullet$};
		\node at (1.5,3.1) {$\bullet$};	
		\node at (2.5,1.6) {$\bullet$};
		\node at (2.5,2.5) {$\bullet$};
		\draw [dotted] (0.5,0) -- (0.5,4);
		\draw [dotted] (1.5,0) -- (1.5,4);
		\draw [dotted] (2.5,0) -- (2.5,4);
		\node at (2,0) [below] {$\bI^k = \bI$};
		\node at (0,2) [left] {$\bI^m = \bI$};
	\end{tikzpicture}
	\caption{An element of $F_8(\I^k \times \I^m)$ for $k=1$, $m=1$ (suppressing the labels on the point for the sake of clarity) which is in $F^{[r]}_8(\bI^k \times \bI^m)$ when $r \geq 4$, but not when $r<4$.}
	\label{fig:examfr}
\end{figure}
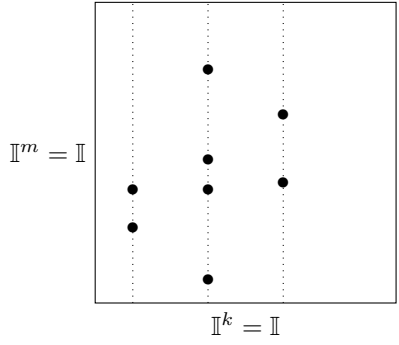
	
Since the condition defining $F^{[r]}_n(\I^k \times \I^m)$ is invariant under the $\fS_n$-action, these topological spaces may be assembled into a symmetric sequence $F^{[r]}(\I^k \times \I^m) \subset F(\I^k \times \I^m)$ in $\cat{sSet}$ and by the copowering also in $\cat{S}^\bN$. The left $\cE_{k+m}$-module structure on $F(\I^k \times \I^m)$ does not restrict. However, using the map of operads $\cE_k \to \cE_{k+m}$ induced by the inclusion $\I^k \to \I^k \times \I^m$ on the first $k$ coordinates, we get a left $\cE_k$-module structure on $F(\I^k \times \I^m)$ which \emph{does} restrict and application of this symmetric sequence gives a functor $\bfF^{[r]}(\I^k \times \I^m)  \colon \cat{S}^\bN \to \cat{Alg}_{\cE_k}(\cat{S}^\bN)$.

As we assumed that $X$ is cofibrant, we can use a monadic bar resolution to give an explicit formula for $\bfT^\bL_r(\cU\bfE_{k+m}(X)) \in \cat{Alg}_{\cE_k}(\cat{S}^\bN)$:
\[\bfT^\bL_r(\cU\bfE_{k+m}(X)) = B(\bfE_k(\leq r)_*,\cE_k,(\leq r)^* \cU\bfE_{k+m}(X)).\]
We take this specific model for the domain of the map in the following proposition:

\begin{proposition}\label{prop:alpha-r} There are weak equivalences
\[\alpha_r \colon \bfT^\bL_r(\cU\bfE_{k+m}(X)) \lra \bfF^{[r]}(\I^k \times \I^m)(X),\]
of $\cE_k$-algebras, which fit into commutative diagrams for $r \geq 0$
\begin{equation}\label{eqn:comm-alpha-r} \begin{tikzcd} \bfT^\bL_r(\cU\bfE_{k+m}(X)) \rar \dar{\alpha_r} & \bfT^\bL_{r+1}(\cU \bfE_{k+m}(X)) \dar{\alpha_{r+1}} \\
 \bfF^{[r]}(\I^k \times \I^m)(X) \rar &  \bfF^{[r+1]}(\I^k \times \I^m)(X).\end{tikzcd}\end{equation}
\end{proposition}

Let us start by defining the maps:

\begin{lemma} There are maps $\alpha_r \colon \bfT^\bL_r(\cU\bfE_{k+m}(X)) \lra \bfF^{[r]}(\I^k \times \I^m)(X)$ of $\cE_k$-algebras making Diagram (\ref{eqn:comm-alpha-r}) commute.
\end{lemma}

\begin{proof}The map $\cE_{k+m} \to F(\I^k \times \I^m)$ which sends a cube to its center is a homotopy equivalence of left $\cE_{k+m}$-modules in symmetric sequences, so we have an induced weak equivalence $\bfE_{k+m}(X) \to \bfF(\bI^k \times \bI^m)(X)$ of $\cE_{k+m}$-algebras. To define $\alpha_r$, we first insert this weak equivalence into the right entry of the bar construction
\[\begin{tikzcd}{|B_\bullet(\bfE_k(\leq r)_*,\cE_k,(\leq r)^* \cU\bfE_{k+m}(X))|} \dar{\simeq} \\ {|B_\bullet(\bfE_k(\leq r)_*,\cE_k,(\leq r)^* \cU\bfF(\I^k \times \I^m)(X))|}.\end{tikzcd}\]

The assumption that $X$ is concentrated in rank $1$ gives us an isomorphism 
\[(\leq r)^*\cU\bfF(\I^k \times \I^m)(X)) \cong \cU((\leq r)^* \bfF(\I^k \times \I^m))((\leq r)^* X).\]
Here $(\leq r)^* F(\I^k \times \I^m)$ is an object in the truncated symmetric sequence category  (functors from the category of possibly empty finite sets of cardinality $\leq r$ into $\cat{sSet}$ with tensor product the restriction of the composition product) and $(\leq r)^* \bfF(\I^k \times \I^m)$ is the functor given by tensoring with $(\leq r)^* F(\I^k \times \I^m)$.

Because $\otimes$ commutes with colimits in each variable and geometric realization, the target is obtained by applying the symmetric sequence
\[|B_\bullet(\cE_k(\leq r)_*,\cE_k,(\leq r)^* F(\bI^{k+m})|\]
in $\cat{Top}$ to $X$ (as always, via $\mr{Sing}$ and the simplicial copowering). We define a map of left $\cE_k$-modules in symmetric sequences in $\cat{Top}$
\[a_r \colon |B_\bullet(\cE_k (\leq r)_*,\cE_k,(\leq r)^* F(\bI^{k+m}))| \lra F^{[r]}(\I^k \times \I^m)\]
by describing an augmentation from $B_\bullet(\cE_k (\leq r)_*,\cE_k,(\leq r)^* F(\bI^{k+m}))$ to $F^{[r]}(\I^k \times \I^m)$. The $0$-simplices of the former are given by
\[\cE_k(\leq r)_*(\leq r)^* F(\bI^{k+m}) = \bigsqcup_{n \geq 1} \cE_k(n) \times_{\fS_n} \bigsqcup_{1 \leq k_1 ,\ldots,k_n\leq r} F_{k_i}(\I^k \times \I^m)\]
where the rank of each component is $k_1 + ...+k_n$. Given a collection of embeddings $e_i \colon \I^k \to \I^k$ and configurations $\xi_i \in F_{k_i}(\I^k \times \I^m)$, we may take the union of the images $(e_i \times \mr{id}_{\I^m})(\xi_i)$ in $\I^k \times \I^m$ and obtain an ordered configuration of $k_1+\ldots+k_m$ points such that no subset $\{x\} \times \bI^m$ contains more than $r$ points. This map is easily seen to be compatible with the left $\cE_k$-module structures. That the diagram commutes is clear from the definition.\end{proof}

We next prove that each $a_r$ is a weak homotopy equivalence, using a microfibration argument.

\begin{definition}
	A map $\pi \colon E \to B$ of topological spaces is a \emph{microfibration} if for each $i \geq 0$ and commutative diagram
	\[\begin{tikzcd} D^i \times \{0\} \rar[hook]{h} \dar & E \dar{\pi} \\
	D^i \times [0,1] \rar{H} & B, \end{tikzcd}\]
	there exists an $\epsilon>0$ and a partial lift $\tilde{H} \colon D^i \times [0,\epsilon] \to E$, i.e.\ $\pi \circ \tilde{H} = H|_{D^i \times [0,\epsilon]}$ and $\tilde{H}|_{D^i \times \{0\}} = h$.
\end{definition}

\begin{lemma}[Lemma 2.2 of \cite{weissclassify}] 	\label{lem:microW}
	If $\pi \colon E \to B$ is a microfibration with weakly contractible fibers, then $\pi$ is a weak homotopy equivalence.
\end{lemma}

Our strategy is to prove that $a_r$ is a microfibration with weakly contractible fibers. To do this, we use the following lemma in point-set topology.

\begin{lemma}\label{lem:point-set} Let $X_\bullet$ be a levelwise Hausdorff simplicial space. Let 
	\[X_{1,\bullet} \subset X_{2,\bullet} \subset \cdots \subset X_{\bullet}\]
be an $\bN_{>0}$-indexed sequence of simplicial subspaces such that: (i) $X_{s,p} \subset X_p$ is compact for all $s,p$, and (ii) each point $x \in X_p$ has an open neighborhood contained in some $X_{s,p}$. If $C$ is compact, then any continuous map $C \to |X_\bullet|$ factors as $C \to |X_{s,\bullet}| \to |X_\bullet|$ for some $s$.
\end{lemma}

\begin{proof}The strategy is to first identify $|X_\bullet|$ with the sequential colimit $\mr{colim}_s \, |X_{s,\bullet}|$ and then show that this particular sequential colimit commutes with maps out of the compact space $C$.
	
The inclusions $X_{s,p} \to X_p$ induce a continuous bijection $\mr{colim}_s \, X_{s,p} \to X_p$. To show it is a homeomorphism we need to prove it is open: $V \subset \mr{colim}_s \, X_{s,p}$ being open means that all $V \cap X_{s,p}$ are open, and by the hypothesis for all $x \in V$, $V$ contains an open neighborhood of $x$ in $X_p$, which means it is open in $X_p$. Since colimits of simplicial spaces are computed levelwise, $\mr{colim}_s \, X_{s,\bullet} \to X_\bullet$ is an isomorphism of simplicial spaces. Since geometric realization commutes with filtered colimits (it has a right adjoint when working with compactly generated weakly Hausdorff spaces), the canonical map $\mr{colim}_s \, |X_{s,\bullet}| \to |X_\bullet|$ is a homeomorphism.

In CGWH spaces, maps out of a compact space commute with sequential colimits of closed inclusions by Lemma 3.6 of \cite{stricklandcgwh}. Thus we shall verify that each map $|X_{s,\bullet}| \to |X_{s+1,\bullet}|$ is a closed inclusion, using its description as a colimit of the maps of skeleta:
	\[\begin{tikzcd} {\mr{sk}_0 |X_{s,\bullet}|} \dar \rar &  {\mr{sk}_1 |X_{s,\bullet}|} \dar \rar & \cdots \\
{\mr{sk}_0 |X_{s+1,\bullet}|} \rar &  {\mr{sk}_1 |X_{s+1,\bullet}|} \rar & \cdots. \end{tikzcd}\]

We claim all maps in this diagram are closed inclusions. All maps are clearly continuous injections and a continuous injection between compact Hausdorff spaces is always a closed inclusion, so it suffices to prove that each space is compact Hausdorff. They are compact because each $\mr{sk}_p|X_{s,\bullet}|$ is a quotient of the compact space $\bigsqcup_{k \leq p} \Delta^k \times X_{s,k}$. They are Hausdorff because we may freely add degeneracies to write $\mr{sk}_p|X_{s,\bullet}|$ as the geometric realization of a levelwise Hausdorff simplicial space and apply Theorem 1.1 of \cite{pazzis}. Furthermore, from the construction it is clear each square is a pullback square. The result then follows from the following result about CGWH spaces, Lemma 3.9 of \cite{stricklandcgwh}: given a commutative diagram
\[\begin{tikzcd}A_0 \rar \dar{f_0} & A_1 \rar \dar{f_1} & \cdots \\
B_0 \rar & B_1 \rar & \cdots \end{tikzcd}\]
with all maps closed inclusions and all squares pullbacks, the induced map $\mr{colim}_s \, A_s \to \mr{colim}_s \, B_s$ is also a closed inclusion.\end{proof} 

\begin{lemma}
	The map $a_r \colon |B_\bullet(\cE_k (\leq r)_*,\cE_k,(\leq r)^* F(\bI^{k+m})| \to F^{[r]}(\I^k \times \I^m)$ is a microfibration. 
	\label{ismicro}
\end{lemma}

\begin{proof}Fixing a cardinality $n$, we need to prove that the component $a_r(n)$ is a microfibration. Suppose we are given a commutative diagram
	\[\begin{tikzcd} D^i \times \{0\} \dar \rar{h} & {|B_\bullet(\cE_k (\leq r)_*,\cE_k,(\leq r)^* F(\bI^{k+m}))|}(n) \dar{a_r(n)} \\
	 D^i \times [0,1] \rar{H} & F^{[r]}_n(\I^k \times \I^m).\end{tikzcd}\]	 
	 
	Since $D^i \times [0,1]$ is compact, there exists a $\delta>0$ such that $H$ factors over the compact subspace of configurations $\xi$ where
		\begin{enumerate}[\indent (a)]
		\item  the points in $\xi$ have distance $\geq \delta$ from each other,
		\item for all closed cubes $C \subset \bR^k$ with equal sides of length $<\delta$, the set $C \times \I^m$ contains at most $r$ points of $\xi$.
	\end{enumerate}
	
	Let us abbreviate $B_p(\cE_k (\leq r)_*,\cE_k,(\leq r)^* F(\bI^{k+m}))(n)$  by $X_p$, and by $X^\delta_p$ the subspace of $X_p$ of elements whose image under $\alpha_r(n)$ satisfies (a) and (b). 
	
	Let $\rho_p \colon X^\delta_p \to (0,\infty)$ be the minimum of the distances from the points in the image $\xi$ to the boundaries of the images of the cubes. Then $X^\delta_\bullet$ is a simplicial space with a sequence of continuous functions $\rho_p \colon X^\delta_p \to (0,\infty)$ such that $\rho_{p+1} \circ s_i = \rho_p$ and $\rho_{p-1} \circ d_i \geq \rho_p$. For each integer $s \geq 1$, the subspaces $X^\delta_{s,p} \coloneqq \rho_p^{-1}([1/s,\infty)) \subset X^\delta_p$ assemble to a simplicial space $X^\delta_{s,\bullet}$. 
	
	This satisfies the hypotheses of Lemma \ref{lem:point-set} (condition (i) of that lemma is the reason we use $X^\delta_\bullet$ instead of $X_\bullet$, and uses that only the interiors of cubes need to be disjoint, not their closures). 	Hence the map $h$ factors over some stage $|X^\delta_{s,\bullet}|$ with $\delta \geq \frac{1}{s}>0$ such that for all $d \in D^i$, the configuration $H(d,0)$ is given by a $\xi$ which satisfies the properties
	\begin{enumerate}[\indent (a)]
		\item the points in $\xi$ have distance $\geq \frac{1}{s}$ from each other,
		\item for all closed cubes $C \subset \bR^k$ with equal sides of length $<\frac{1}{s}$, the set $C \times \I^m$ contains at most $r$ points of $\xi$,
		\item the points in $\xi$ have distance $\geq \frac{1}{s}$ to the boundaries of the images of the cubes in $h(d)$.
	\end{enumerate}
	
	By continuity of the map $H$, there is an $\epsilon>0$ such that for all $d \in D^i$ and $t \in [0,\epsilon]$, the configuration $H(d,t)$ is within distance $\frac{1}{3s}$ of $H(d,0)$. The partial lift is given by the map which assigns to $(d,t) \in D^i \times [0,\epsilon]$ the element of $|X_\bullet|$ represented by configuration $H(d,t)$ inside the cubes coming from the unique non-degenerate representative of $h(d)$. To see this is well-defined, note that (c) implies a point in $H(d,t)$ remains within the same cubes of $h(d)$ as the corresponding point in $H(d,0)$.
	
	To see it is continuous, note that the movement of the points in the configuration can be described by recording their displacements by an element $\Delta(d,t)$ of $\left([-\frac{1}{3s},\frac{1}{3s}]^{k+m}\right)^n$ (each $\Delta(d,0)$ equals $0$). That is, $\Delta(d,t)$ is defined by $H(d,t) = H(d,0) +\Delta(d,t)$.
	
	There is a simplicial map
	\[\left( [-\frac{1}{3s},\frac{1}{3s}]^{k+m}\right)^n \times X^{1/s}_{s,\bullet} \lra X_\bullet\]
	obtained by applying the displacement to the configuration. This is continuous, and well-defined because whenever we move points in the configuration of an element of $X^{1/s}_{s,p}$ at most $\frac{1}{3s}$ in any of the directions, they do not (a) collide with each other, (b) have more than $r$ points in a subset $\{x\} \times \I^m$, and (c) cross boundaries of cubes. That the lift is continuous then follows by observing that it can be realized as a composition of continuous maps
	\begin{align*}D^i \times [0,\epsilon] &\overset{\Delta \times h}\lra \left( [-\frac{1}{3s},\frac{1}{3s}]^{k+m}\right)^n \times |X^{1/s}_{s,\bullet}| \\
	&\overset{\cong}{\lra} \left| \left( [-\frac{1}{3s},\frac{1}{3s}]^{k+m}\right)^n \times X^{1/s}_{s,\bullet} \right| \\
	&\lra |X_\bullet|.\qedhere\end{align*}
\end{proof}

	We now prove that the fibers of $a_r(n)$ are weakly contractible, so Lemma \ref{lem:microW} implies that $a_r(n)$ is a weak homotopy equivalence.
	
	\begin{lemma} The fibers of the map $a_r \colon |B_\bullet(\cE_k (\leq r)_*,\cE_k,(\leq r)^* F(\bI^{k+m})| \to F^{[r]}(\I^k \times \I^m)$ are weakly contractible.
	\label{iscontractible}
	\end{lemma}
	\begin{proof}

\begin{figure}
	\begin{tikzpicture}
	\draw (0,0) rectangle (4,4);
	\draw [fill=blue!5!white] (0.25,0) rectangle (3,4);
	\draw [fill=blue!10!white] (0.5,0) rectangle (1.5,4);
	\draw [fill=blue!10!white] (2,0) rectangle (2.75,4);
	\node at (1,.3) {$\bullet$};
	\node at (1,1.5) {$\bullet$};	
	\node at (1,1.9) {$\bullet$};
	\node at (1,3.1) {$\bullet$};	
	\node at (2.5,1.6) {$\bullet$};
	\node at (2.5,2.5) {$\bullet$};
	\draw [dotted] (1,0) -- (1,4);
	\draw [dotted] (2.5,0) -- (2.5,4);
	\node at (2,0) [below] {$\bI^k = \bI$};
	\node at (0,2) [left] {$\bI^m = \bI$};
	\end{tikzpicture}
	\caption{An element of $X_1$ in the case $k=m=1$, $r=4$, and $n=6$. There are two innermost cubes and one outermost cube.}
	\label{fig:microfib}
\end{figure}
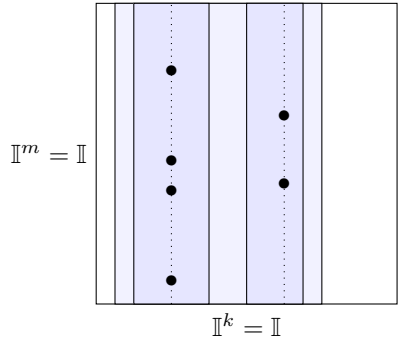

	Fix a configuration $\xi \in \smash{F^{[r]}_n}(\I^k \times \I^m)$. As above, the fiber $\epsilon^{-1}(\xi)$ is given by the geometric realization of the subsimplicial space of $X_\bullet$ with underlying configuration $\xi$. Call this simplicial space $X_\bullet(\xi)$. Another application of Lemma 9.14 of \cite{GKRW1} tells us this is Reedy cofibrant.
	
	Let $\xi' \in \mr{Sym}_n(\I^k) \coloneqq (\I^{k})^n/\fS_n$ be the configuration with multiplicities obtained by projecting $\xi$ onto $\I^k$. The $p$-simplices of $X_\bullet(\xi)$ are given by $p+1$ levels of nested $k$-dimensional cubes such that all points of $\xi'$ are contained in an innermost cube and all cubes except the outermost ones contain at most $r$ points of $\xi'$ counted with multiplicity. Let $(e_1,e_2, \ldots e_l) \in \cE_k(l)$ be a collection of cubes such that every cube $e_i$ contains exactly one point of $\xi'$ (counted without multiplicity) and every point of $\xi'$ (counted without multiplicity) is in one of the cubes $e_i$. Let $X_\bullet(\xi, e)$ denote the subsimplicial space of $X_\bullet(\xi)$ where we require that if a cube contains a point of $\xi'$, then it contains the corresponding $e_i$. This is also Reedy cofibrant. Thus, since the inclusion $X_\bullet(\xi, e) \hookrightarrow X_\bullet(\xi)$ induces a levelwise homotopy equivalence, it induces a weak equivalence on geometric realizations. View $X_\bullet(\xi, e)$ as an augmented simplicial space by adding a point in degree $-1$.  There is an extra degeneracy $X_p(\xi, e) \to X_{p+1}(\xi, e)$ given by inserting $e_i$ in the innermost cubes, and hence $|X_\bullet(\xi, e)|$ is contractible.

\end{proof}

\begin{proof}[Proof of Proposition \ref{prop:alpha-r}] By combining Lemmas \ref{lem:microW}, \ref{ismicro}, and  \ref{iscontractible}, we see that the map $a_r \colon |B_\bullet(\cE_k (\leq r)_*,\cE_k,(\leq r)^* F(\bI^{k+m})| \to F^{[r]}(\I^k \times \I^m)$ is a weak equivalence. Since $a_r$ is a map of symmetric sequences and all of the symmetric group actions are free, it is weak equivalence of symmetric sequences. The result follows because applying a weak equivalence between $\Sigma$-cofibrant symmetric sequences in $\cat{sSet}$ to a cofibrant object is a weak equivalence by Lemma 9.1 of  \cite{GKRW1}, and geometric realization preserves weak equivalences between Reedy cofibrant simplicial objects.	
\end{proof}

The inclusion $\I^k \times F_{r+1}(\I^m) \hookrightarrow F_{r+1}(\I^k \times \I^m)$ given by $(x,\xi) \mapsto x \times \xi$, i.e.\ sending each point $m_i \in \xi$ to $x \times m_i$, has image given by the complement of $\smash{F_{r+1}^{[r]}}(\I^k \times \I^m)$ in $F_{r+1}(\I^k \times \I^m)$. Let $\varphi_{m,r}$ denote the trivial vector bundle over $\I^k \times F_r(\I^m)$ given by $\I^k \times F_r(\I^{m}) \times \bR^{r-1} \to \I^k \times F_r(\I^m)$, with $\fS_r$ acting diagonally and with $\bR^{r-1}$ the orthogonal complement to the trivial representation in the permutation representation with its usual metric. The vector bundle $\varphi_{m,r}$ can be thought of as the $\fS_r$-equivariant analogue of $\phi_{m,r}$ from the introduction.

\begin{lemma} \label{lem:thomspacelevel2}
	The normal bundle of $\I^k \times F_{r+1}(\I^m)$ in $F_{r+1}(\I^k \times \I^m)$ is $\fS_{r+1}$-equivariantly isomorphic to $k\varphi_{m,r+1}$.
\end{lemma}

\begin{proof}The normal bundle to $\I^k \times F_{r+1}(\I^m)$ is the orthogonal complement in $T(F_{r+1}(\I^k \times \I^m))$ to the tangent bundle $T(\I^k \times F_{r+1}(\I^m))$. The former is the restriction of $T(\I^k \times \I^m)^{\oplus r+1} \cong (T\bI^k)^{\oplus r+1} \oplus (T\I^m)^{\oplus r+1}$, and the latter is the restriction of $T\bI^k \oplus (T\I^m)^{\oplus r+1}$. The inclusion is the diagonal on the first term and the identity on the second, and equivariant for the $\fS_{r+1}$-action. Thus the normal bundle is $\fS_{r+1}$-equivariantly isomorphic to the restriction of the orthogonal complement of the diagonal $T\bI^k \subset (T\bI^k)^{\oplus r+1}$. This is isomorphic to a $k$-fold Whitney sum of the trivial $\bR^{r}$-bundle, with $\fS_{r+1}$-action given by the standard representation.\end{proof}

The vector bundle $\varphi_{m,r+1}$ inherits a Riemannian metric, and we let $S(k\varphi_{m,r+1})$ be the sphere bundle with fiber over $(x,\xi) \in \I^k \times F_{r+1}(\I^{m})$ those vectors of length $\frac{1}{2}d(y,\partial I^k)$. This bounds a disk bundle $D(k\varphi_{m,r+1})$, and both are clearly isomorphic to the unit sphere and disk bundles. Using the exponential map, we obtain the horizontal maps in the following commutative diagram of $\fS_{r+1}$-spaces:
\[\begin{tikzcd} S(k\varphi_{m,r+1}) \rar \dar & F^{[r]}_{r+1}(\I^k \times \I^m) \dar \\
D(k\varphi_{m,r+1}) \rar & F^{[r+1]}_{r+1}(\I^k \times \I^m) = F_{r+1}(\I^k \times \I^m). \end{tikzcd}\]

\begin{proposition}\label{prop:pushout-free}For all $r \geq 0$, there is a zigzag of homotopy cocartesian squares
	\[\begin{tikzcd} \bfE_k ((r+1)_* S(k \varphi_{m,r+1}) \otimes_{\fS_{r+1}} X^{\otimes r+1})  \dar \rar &[-6pt] \dar \cdots &[-6pt] \lar{\simeq} \rar \dar \cdots &[-6pt] \bfT^\bL_{r}(\cU \bfE_{k+m}(X)) \dar \\
	\bfE_k ((r+1)_*  D(k \varphi_{m,r+1}) \otimes_{\fS_{r+1}} X^{\otimes r+1}) \rar & \cdots & \lar{\simeq} \cdots \rar & \bfT^\bL_{r+1}(\cU \bfE_{k+m}(X)). \end{tikzcd}\]
\end{proposition}

\begin{proof}The above result implies that for cofibrant $X$, we have a homotopy cocartesian square
	\[\begin{tikzcd} S(k\varphi_{m,r+1}) \otimes_{\fS_{r+1}} X^{\otimes r+1} \rar \dar &  \dar (r+1)^* \cU^{\cE_k} \bfF^{[r]}(\I^k \times \I^m)(X) \\
	D(k\varphi_{m,r+1})\otimes_{\fS_{r+1}} X^{\otimes r+1}	 \rar & (r+1)^* \cU^{\cE_k} \bfF^{[r+1]}(\I^k \times \I^m)(X).\end{tikzcd}\]
By Proposition \ref{prop:alpha-r}, we have a commutative diagram with horizontal maps weak equivalences
\[\begin{tikzcd} (r+1)^* \cU^{\cE_k}  \bfF^{[r]}(\I^k \times \I^m)(X)   \dar & (r+1)^* \cU^{\cE_k}  \bfT^\bL_r(\cU\bfE_{k+m}(X)) \dar \lar{\alpha_r}[swap]{\simeq} \\
(r+1)^* \cU^{\cE_k}  \bfF^{[r+1]}(\I^k \times \I^m)(X) &  (r+1)^* \cU^{\cE_k}  \bfT^\bL_{r+1}(\cU \bfE_{k+m}(X)) \lar{\alpha_{r+1}}[swap]{\simeq}.\end{tikzcd}\]

Since applying $\bfE_k$ and $(r+1)_*$ preserves homotopy cocartesian squares, doing so gives us the left and middle squares. For the right square, specialize Proposition \ref{prop:attach-one-step} to $\cO = \cE_k$ and $\bfA = \cU \bfE_{k+m}(X)$ to obtain a homotopy cocartesian square	\[\begin{tikzcd}\bfE_k((r+1)_* (r+1)^* \cU^{\cE_k} \bfT^\bL_r(\cU \bfE_{k+m}(X))) \dar \rar &   \bfT^\bL_r(\cU \bfE_{k+m}(X)) \dar \\
\bfE_k((r+1)_* (r+1)^* \cU^{\cE_{k}} \bfT^\bL_{r+1}(\cU \bfE_{k+m}(X))) \rar & \bfT^\bL_{r+1}(\cU \bfE_{k+m}(X)). \end{tikzcd}\qedhere \]
\end{proof}

We now deduce Theorem \ref{thm:mainhighercell} from this by taking $\cat{S} = \cat{Top}$ and $X = (1)^*(\ast)$; we need to resolve the issue that Proposition \ref{prop:pushout-free} only provides zigzags.

\begin{proof}[Proof of Theorem \ref{thm:mainhighercell}]

We start with an elementary homotopy-theoretic observation. Given a commutative diagram of topological spaces 
\[\begin{tikzcd} S(k\phi_{m,r+1}) \dar \rar & X \dar & \lar{\simeq} X' \dar \\
D(k\phi_{m,r+1}) \rar & Y & \lar{\simeq} Y' \end{tikzcd}\]
with decorated arrows weak equivalences, we can find maps $S(k\phi_{m,r+1}) \to X'$ and $D(k\phi_{m,r+1}) \to Y'$ such that in the following diagram
\[\begin{tikzcd} S(k\phi_{m,r+1}) \dar \rar \arrow[bend left=30]{rr} & X \dar & \lar{\simeq} X' \dar \\
D(k\phi_{m,r+1}) \rar \arrow[bend right=30]{rr} & Y & \lar{\simeq} Y' \end{tikzcd}\]
the outer square commutes and the triangles commute up to homotopy. To prove this, first homotope $S(k\phi_{m,r+1}) \to X$ until a lift exists (which is possible since the domain has the homotopy type of a CW-complex). Because $S(k\phi_{m,r+1}) \hookrightarrow D(k\phi_{m,r+1})$ admits the structure of a NDR-pair, we may extend this to a homotopy of commutative diagrams. At this point it suffices to find a lift in the commutative diagram
\[\begin{tikzcd}S(k\phi_{m,r+1}) \dar \rar & Y' \dar \\
D(k\phi_{m,r+1}) \rar & Y,\end{tikzcd}\]
which exists as $(D(k\phi_{m,r+1}),S(k\phi_{m,r+1}))$ is homotopy equivalent to a CW pair.

\vspace{.5em}

Given this observation, we prove by induction over $r$ that we may construct $\bfA_r \simeq \bfT^\bL_{r+1}(\cU \bfE_{k+m}(\ast))$ by iterated pushouts along free algebras, obtaining in the process maps between the $\bfA_r$ satisfying $\mr{colim}_r \, \bfA_r = \mr{hocolim}_r \, \bfA_r \simeq \bfF^{\cE{k+m}}(\ast)$. 

The initial case is $\bfA_{-1} = \varnothing$. For the induction step, let us assume we have produced $\bfA_r$ as in the statement of Theorem \ref{thm:mainhighercell} with a weak equivalence $\beta_r \colon \bfA_r \to \bfT^\bL_r(\cU \bfE_{k+m}(\ast))$. Using the observation in the diagram of Proposition \ref{prop:pushout-free}, we may assume we have a homotopy cocartesian commutative diagram
	\[\begin{tikzcd} S(k\phi_{m,r+1}) \dar \rar & (r+1)^* \cU^{\cE_k}  \bfT^\bL_r(\cU\bfE_{k+m}(\ast)) \dar \\
D(k\phi_{m,r+1}) \arrow{r} & (r+1)^* \cU^{\cE_k}  \bfT^\bL_{r+1}(\cU \bfE_{k+m}(\ast)).\end{tikzcd}\]

Applying the observation again to lift along $\alpha_r$,  we may assume we have a homotopy cocartesian commutative diagram
	\[\begin{tikzcd} S(k\phi_{m,r+1}) \dar \rar  & (r+1)^* \cU^{\cE_k} \bfA_r \rar{(r+1)^* \cU^{\cE_k} \beta_r}  &[20pt] (r+1)^* \cU^{\cE_k}  \bfT^\bL_r(\cU\bfE_{k+m}(\ast)) \dar \\
D(k\phi_{m,r+1}) \arrow{rr} & & (r+1)^* \cU^{\cE_k}  \bfT^\bL_{r+1}(\cU \bfE_{k+m}(\ast)). \end{tikzcd}\]
We take adjoints and define $\bfA_{r+1}$ as the pushout fitting in a commutative diagram
	\[\begin{tikzcd} \bfE_k((r+1)_* S(k\phi_{m,r+1})) \rar \dar & \bfA_r \rar{\beta_r} \dar & \bfT^\bL_r(\cU\bfE_{k+m}(\ast)) \dar \\
\bfE_k((r+1)_* D(k\phi_{m,r+1})) \rar & \bfA_{r+1} \rar{\beta_{r+1}} &  \bfT^\bL_{r+1}(\cU \bfE_{k+m}(\ast)). \end{tikzcd}\]
The outer and left squares are homotopy cocartesian, the former by construction and the latter as a homotopy pushout. Thus the right square is also homotopy cocartesian, and hence the map $\beta_{r+1} \colon \bfA_{r+1} \to \bfT^\bL_{r+1}(\cU\bfF^{\cE_{k+m}}(\ast))$ is a weak equivalence.
\end{proof}

Combining the last sentence of this proof with Proposition \ref{prop:alpha-r}, we get the description of $\bfA_r$ announced in the introduction.

\begin{corollary}\label{cor:a-desc} There are weak equivalences of $\cE_k$-algebras
	\[\bfA_r \overset{\beta_r}\lra \bfT^\bL_r(\cU\bfE_{k+m}(\ast)) \overset{\alpha_r}\lra \bfF^{[r]}(\I^k \times \I^m)(\ast).\]
\end{corollary}

\begin{remark}It is plausible that Theorem \ref{thm:mainhighercell} may be deduced from results analogous to those in \cite{GKRW1}. One would need an CW approximation theorem for $\cE_k$-algebras in $\cat{Top}^\bN$, and verify that one may desuspend the identification of $\Sigma^k Q^{\cE_k}_\bL(\bfF^{\cE_{k+m}}(\ast))$ with the $k$-fold bar construction with respect to the canonical augmentation.
\end{remark}

\section{Relation to the May-Milgram filtration}\label{sec:may-milgram}

We now explain the relationship between the results in the previous section and the May-Milgram filtration on $\Omega^m \Sigma^m S^{k}$.

\begin{definition} \label{defConfig}
Given a based topological space $(X,x_0)$, let $C(M;X)$ be the quotient of $\bigsqcup_{n \geq 0} F_n(M) \times_{\fS_n} X^n$ by the relation that $(m_1, \ldots m_n;x_1,\ldots x_n)$ is equivalent to $(m_1, \ldots m_{n-1};x_1,\ldots x_{n-1})$ if $x_n=x_0$. We call this the \emph{configuration space of unordered points in $M$ with labels in $X$}.
\end{definition}

When $X=S^0$, we recover ordinary unordered configuration spaces and drop $X$ from the notation. Work of Milgram and May implies that $\Omega^m \Sigma^m X$ has the weak homotopy type of $C(\I^m;X)$ when $X$ is connected \cite{Mil2,M}. The $r$th stage $\cM_r (C(\I^m;X))$ of the May-Milgram filtration of  $C(\I^m;X) \simeq \Omega^m \Sigma^m X$ is defined to be the image of $F_r(\I^m) \times_{\fS_i} X^i$ in  $C(\I^m;X)$.

In this paper, we use only the case $X=S^k$. In that case, $\Omega^m \Sigma^m S^{k}$ is weakly equivalent to the $k$-fold delooping of $C(\I^k \times \I^m) = \bfF(\I^k \times \I^m)(\ast) \simeq \bfE_{k+m}(\ast)$. Let us denote the $\cE_k$-algebra $\bfF^{[r]}(\I^k \times \I^m)(\ast)$ by $C^{[r]}(\I^k \times \I^m)$. 

\begin{theorem} \label{mainfilt}
	The $k$-fold delooping of $C^{[r]}(\I^k \times \I^m)$ is homotopy equivalent to the $r$th stage in the May-Milgram filtration of $\Omega^m \Sigma^m S^{k}$. 
\end{theorem}

To prove this, we need to consider a generalization of $C^{[r]}(\I^k \times \I^m)$ where points can vanish if they enter certain regions. 

\begin{definition}
	Let $M$ be a manifold and $N \subset M$ a subspace. Let $C^{[r]}(M \times \I^m)$ denote the subspace of $C(M \times \I^m)$ of configurations $\xi$ where $\xi \cap \left( \{x\} \times \I^m \right)$ has cardinality $\leq r$ for all $x \in M$. Let $C^{[r]}( (M,N) \times \I^m)$ be the quotient of $C^{[r]}(M \times \I^m)$ by the equivalence relation that $\xi \sim \xi'$ if $\xi \cap(  (M \setminus N) \times \I^m) =\xi' \cap ( ( (M \setminus N) \times \I^m)$.
\end{definition}

We drop the superscript for $r=\infty$ and drop the $- \times \I^m$ for $m=0$. There are two configuration space models for  $\Omega^m \Sigma^m S^{k} = \Omega^m S^{k+m}$. The first is a special case of May's approximation theorem from \cite{M}, building on the work of Milgram in \cite{Mil2}, and the second is a specialization of Proposition 2 of \cite{B}.

\begin{theorem}[May] \label{mayscan}
	For $k>0$, $C(\I^m;S^k)$ is weakly homotopy equivalent to $\Omega^m \Sigma^m S^{k}$.
\end{theorem} 

\begin{theorem}[B\"odigheimer] \label{scanning}
	For $k>0$, $C( (\bR^k,\bR^k \setminus \I^k) \times \I^m)$ is weakly homotopy equivalent to $\Omega^m \Sigma^m S^{k}$.
\end{theorem}

We will relate these two models of $\Omega^m \Sigma^m S^{k}$, and compare filtrations of these spaces. The topological space $C( (\bR^k,\bR^k \setminus \I^k) \times \I^m)$ is filtered by the $C^{[r]}( (\bR^k,\bR^k \setminus \I^k) \times \I^m)$. From now on, we view $S^k$ as $\bR^k/(\bR^k \setminus \I^k)$ with base point given by the image of $\bR^k \setminus \I^k$. Define a map $\rho$ by
\begin{align*}\rho \colon C(\I^m;S^k) &\lra C( (\bR^k,\bR^k \setminus \I^k) \times \I^m) \\
((m_1;x_1),\ldots, (m_r,x_r)) &\longmapsto (x_1 \times m_1,\ldots, x_r \times m_r),\end{align*}
where $m_i \in \I^m$ and $x_i \in \bR^k/(\bR^k \setminus \I^k) = S^k$. This inclusion has image consisting of those configurations with at most one point in each fiber of $  \bR^k \times \I^m \to \I^m$. We denote its restrictions by $\rho_r \colon \mathcal \cM_r (C(\I^m;S^k)) \to C^{[r]}( (\bR^k,\bR^k \setminus \I^k) \times \I^m)$.

\begin{lemma} \label{filtrations}
	The maps $\rho$ and $\rho_r$ are homotopy equivalences. 
\end{lemma}

\begin{proof}The strategy is to scale the configurations so that in each fiber of $ \bR^k \times \I^m \to \I^m$ all but at most one point is pushed into $\bR^k \setminus \I^k$. To do so, we pick a continuous function $\eta \colon C((\bR^k,\bR^k \setminus\I^k) \times \I^m) \to (0,\infty)$ with the property that for all $\xi \in C( (\bR^k,\bR^k \setminus\I^k) \times \I^m)$ and $x \in \I^m$, there is at most one point in $\xi$ that is within distance $\eta(\xi)$ of $0 \times x \in \bR^k \times \I^m$. Let $\phi_t^R \colon \bR^k \to \bR^k$ be a continuous family of maps, depending on $t \in [0,1]$ and $R >0$, such that:
	\begin{itemize}
		\item $\phi^R_0 = \mr{id}$,
		\item $\phi^R_t|_{(\phi^R_t)^{-1}( \I^k)  }$ a homeomorphism onto its image,
		\item $\phi_t^R( \bR^k \setminus \I^k) \subset \bR^k \setminus \I^k$ and $\phi^R_1(y) \in  \bR^k \setminus \I^k$ if $||y||>R$. 
	\end{itemize}
	Then we define 
	\begin{align*} H \colon [0,1] \times C((\bR^k,\bR^k \setminus \I^k) \times \I^m) &\lra C( (\bR^k,\bR^k \setminus\I^k) \times \I^m) \\
	(t,\xi) &\longmapsto (\mr{id} \times \phi_t^{\eta(\xi)})_*(\xi),\end{align*}
	where the subscript $*$ means induced map on configuration spaces. For $t=1$, all but at most one point in each fiber are pushed into $\bR^k \setminus \I^k$ (where these points vanish). In particular, we can regard it as a continuous map 
	\[h \colon C( (\bR^k,\bR^k \setminus \I^k) \times \I^m) \lra C(\I^m;S^k).\]
	The homotopy $H$ then provides a homotopy from $\rho \circ h$ to the identity on $C( (\bR^k,\bR^k \setminus \I^k) $, and since it preserves the subspace $C(\I^m;S^k)$ also a homotopy from $h \circ \rho$ to the identity on $C(\I^m;S^k)$. Thus $\rho$ is a homotopy equivalence.
	
	Since $\rho$, $h$ and $H$ preserve the filtration, this also proves the $\rho_r$ are homotopy equivalences.
\end{proof}

Thus $C^{[r]}( (\bR^k,\bR^k \setminus \I^k) \times \I^m) $ is homotopy equivalent to the $r$th stage of the May-Milgram filtration. We claim that the $k$-fold delooping of $C^{[r]}(\I^k \times \I^m) $ is $C^{[r]}( (\bR^k,\bR^k \setminus \I^k) \times \I^m)$. 

The $k$-fold bar construction of an augmented $E_k$-algebra is defined in full generality in Section 13.1 of \cite{GKRW1}. We will specialize it to the $\cE_k$-algebra $C^{[r]}(\bI^k \times \bI^m)$ in $\cat{Top}$, with its canonical augmentation to $\ast$, and make a minor modification to the ``grids'' for the sake of computational convenience, replacing $[0,1]$ with $[1/4,3/4]$ in the following definition:

\begin{definition}We write $\cP_k(p_1,\ldots,p_k) \subset \prod_j \bR^{p_j+1}$ for the subspace of $k$-tuples \[\{1/4<t^j_0<\cdots<t^j_{p_j}<3/4\}_{1 \leq j \leq k}.\] We make $[p_1,\ldots,p_k] \mapsto \cP_k(p_1,\ldots,p_k)$ into a $k$-fold semi-simplicial space by defining the $i$th face map in the $j$th direction by forgetting $t^j_i$.\end{definition}

\begin{definition}$B^{E_k}_{\bullet,\cdots,\bullet}(C^{[r]}(\bI^k \times \bI^m))$ is the $k$-fold semi-simplicial space with $(p_1,\ldots,p_k)$-simplices given by the subspace of
	\[(\{t^j_i\},\xi) \in \cP_k(p_1,\ldots,p_k) \times C^{[r]}(\bI^k \times \bI^m)\]
such that $\xi$ is contained in $\prod_j [t^j_0,t^j_{p_j}] \times \bI^m $ and $\xi$ is disjoint from $[1/4,3/4]^{j-1} \times \{t^j_i\} \times [1/4,3/4]^{k-j} \times \bI^m$ for $1 \leq j \leq k$ and $0 \leq i \leq p_j$.

For $0<i<p_j$, the $i$th face map in the $j$th direction is given by the corresponding face map on $\cP_k$ and the identity on $C^{[r]}(\bI^k \times \bI^m)$. The $0$th face map in the $j$th direction is given by the corresponding face map on $\cP_k$ and by deleting all particles in $\xi$ which have $j$th coordinate $<t^j_1$. Similarly, the $p_j$th face map in the $j$th direction is given by the corresponding face map on $\cP_k$ and by deleting all particles in $\xi$ which have $j$th coordinate $>t^j_{p_j-1}$.\end{definition}

\begin{definition}The $k$-fold delooping of $C^{r}(\bI^k \times \bI^m)$ is the pointed topological space given by
\[B^kC^{[r]}(\bI^k \times \bI^m) \coloneqq ||B^{E_k}_{\bullet,\cdots,\bullet}(C^{[r]}(\bI^k \times \bI^m))||.\]
\end{definition}

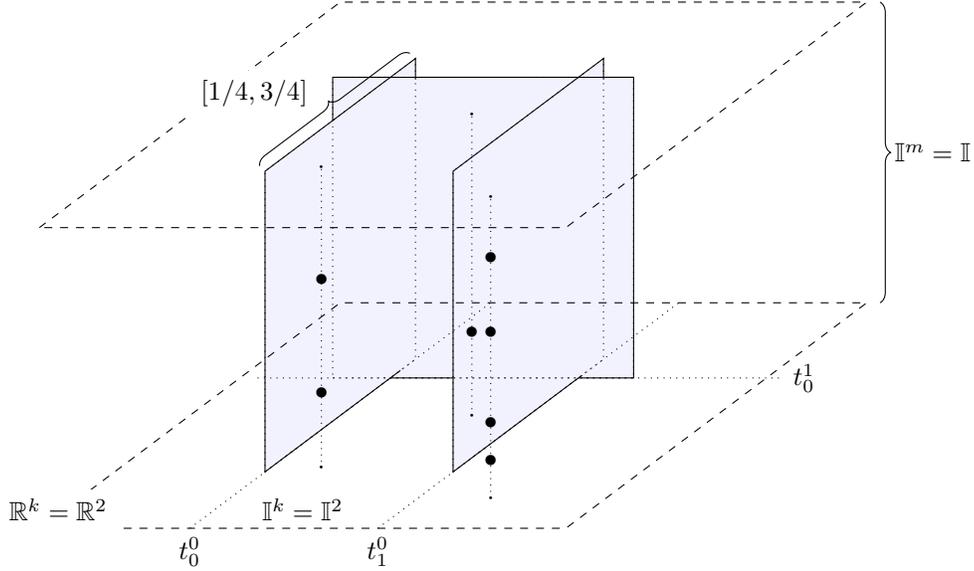
\begin{figure}
	\begin{tikzpicture}
		\draw [xshift=2.9 cm,yshift=2cm,fill=blue!5!white](0,0) -- (4,0) -- (4,4) -- (0,4) -- cycle;
		\draw [xshift=2 cm,yshift=0.75cm,fill=blue!5!white](0,0) -- (2,1.5) -- (2,5.5) -- (0,4) -- cycle;
		\draw [xshift=4.5 cm,yshift=0.75cm,fill=blue!5!white](0,0) -- (2,1.5) -- (2,5.5) -- (0,4) -- cycle;
		\fill [xshift=2.9 cm,yshift=2cm,fill=blue!5!white](3.3,0.03) -- (3.97,0.1) -- (3.97,3.97) -- (3.3,3.97) -- cycle;
		\fill [xshift=2.5 cm,yshift=2cm,fill=blue!5!white](1.3,0.03) -- (1.97,0.1) -- (1.97,3.97) -- (1.3,3.97) -- cycle;
		\draw [xshift=2 cm,yshift=0.75cm,dotted](0,0) -- (2,1.5) -- (2,5.5) -- (0,4) -- cycle;
		\draw [xshift=4.5 cm,yshift=0.75cm,dotted](0,0) -- (2,1.5) -- (2,5.5) -- (0,4) -- cycle;
		\draw [xshift=2.9 cm,yshift=2cm,dotted](0,0) -- (4,0) -- (4,4) -- (0,4) -- cycle;
		\draw [xshift=2.9 cm,yshift=2cm] (4,4) -- (0,4);
		\draw [dashed] (-1,0) -- (3,3) -- (10,3) -- (6,0) -- cycle;
		\node [fill=white] at (-.75,.25)  {$\bR^k = \bR^2$};
		\node at (2.5,.25) {$\I^k = \bI^2$};
		\draw [dashed,yshift=4cm] (-1,0) -- (3,3) -- (10,3) -- (6,0) -- cycle;
		\draw [decorate,decoration={brace,amplitude=4pt},xshift=5pt,yshift=0pt]
		(10,7) -- (10,3) node [black,midway,xshift=.7cm] {$\bI^m = \bI$};
		\draw [decorate,decoration={brace,amplitude=4pt},xshift=-2pt,yshift=2pt]
		(2,4.75) -- (4,6.25) node [black,midway,xshift=-.25cm,yshift=.22cm,left,fill=white] {$[1/4,3/4]$};
		\draw [xshift=2.9 cm,yshift=2cm,dotted] (-1,0) -- (6,0);
		\draw [dotted,xshift=1 cm] (0,0) -- (4,3);
		\draw [dotted,xshift=3.5 cm] (0,0) -- (4,3);
		\node at (8.9,2) [right] {$t^1_0$};
		\node at (1,0) [below] {$t^0_0$};
		\node at (3.5,0) [below] {$t^0_1$};
		\begin{scope}[xshift=2.75 cm,yshift=0.8cm]
		\draw [dotted] (0,0) -- (0,4);
		\node at (0,0) {$\cdot$};		
		\node at (0,4) {$\cdot$};
		\node at (0,1) {$\bullet$};
		\node at (0,2.5) {$\bullet$};
		\end{scope}
		\begin{scope}[xshift=4.75 cm,yshift=1.5cm]
		\node at (0,0) {$\cdot$};		
		\node at (0,4) {$\cdot$};
		\draw [dotted] (0,0) -- (0,4);
		\node at (0,1.1) {$\bullet$};
		\end{scope}
		\begin{scope}[xshift=5 cm,yshift=.4cm]
		\node at (0,0) {$\cdot$};		
		\node at (0,4) {$\cdot$};
		\draw [dotted] (0,0) -- (0,4);
		\node at (0,1) {$\bullet$};
		\node at (0,2.2) {$\bullet$};
		\node at (0,0.5) {$\bullet$};
		\node at (0,3.2) {$\bullet$};
		\end{scope}
	\end{tikzpicture}
	\caption{An element of $X_{1,0}$ for $r=4$, $k=2$, and $m=1$. Points in the configuration disappear when they leave the cube $\bI^2 \times \bI$, cannot hit the walls $\{t^0_i\} \times [1/4,3/4] \times \I$ for $i=0,1$ or $[1/4,3/4] \times \{t^1_0\} \times \I$, and every vertical line segment can contain at most $4$ points.}
\end{figure}

Combined with Lemma \ref{filtrations}, the following proposition completes the proof of Theorem \ref{mainfilt}.

\begin{proposition}	\label{delooponce}
	There is a zig-zag of weak equivalences of $\cE_m$-algebras
	\[B^kC^{[r]}(\bI^k \times \bI^m) \xleftarrow{||f_\bullet||} ||X_{\bullet,\ldots,\bullet}|| \overset{\epsilon}{\lra} C^{[r]}\left((\bR^k,\bR^k \setminus \bI^k) \times \I^m\right).\]
\end{proposition}

\begin{proof} We start by defining the augmented $k$-fold semi-simplicial topological space $X_{\bullet,\cdots,\bullet}$: its topological space of $(p_1,\ldots,p_k)$-simplices 
	\[X_{p_1,\ldots,p_k} \subset \cP_k(p_1,\ldots,p_k) \times C^{[r]}\left((\bR^k,\bR^k \setminus \bI^k) \times \I^m\right)\]
	is the subspace of $(\{t^j_i\},\xi)$ such that $\xi$ is disjoint from $[1/4,3/4]^{j-1} \times \{t^j_i\} \times [1/4,3/4]^{k-j} \times \bI^m$ for each $1 \leq j \leq k$ and $0 \leq i \leq p_j$. This is augmented over $C^{[r]}\left((\bR^k,\bR^k \setminus \bI^k) \times \I^m\right)$. The $i$th face map in the $j$th direction is given by forgetting $t^j_i$ and the augmentation forgets all $t^j_i$'s.
	
	We denote the map 
	\[||X_{\bullet,\cdots,\bullet}|| \lra C^{[r]}\left((\bR^k,\bR^k \setminus \bI^k) \times \I^m\right)\]
	by $\epsilon$. To show this is a weak equivalence, we prove it is a microfibration with weakly contractible fibers and invoke Lemma \ref{lem:microW}. For $\xi \in C^{[r]}\left((\bR^k,\bR^k \setminus \bI^k) \times \I^m\right)$, let $S^j_\xi \subseteq \bR$ be the subspace of $t \in (1/4,3/4)$ such that \[\xi \cap \left([1/4,3/4]^{j-1} \times  \{ t \}  \times [1/4,3/4]^{k-j} \times \I^m\right)  = \varnothing.\] The fiber $\epsilon^{-1}(\xi)$ is the thick geometric realization of a $k$-fold semi-simplicial space with space of $(p_1,\cdots,p_k)$-simplices homotopy equivalent to the product of sets of order preserving-maps from $\{0,\ldots,p_j\}$ to $\pi_0(S^j_\xi)$ for $1 \leq j \leq k$, which is product of simplices. Since levelwise weak equivalences induce weak equivalences on thick geometric realizations (see e.g.\ Theorem 2.2 of \cite{ebertrw}), the fibers of $\epsilon$ are weakly contractible. 
	
	The proof that $\epsilon$ is a microfibration is similar to that of Lemma \ref{ismicro}. The key fact is that if \[\xi \cap \left( [1/4,3/4]^{j-1} \times  \{ t^j\}  \times [1/4,3/4]^{k-j}  \times \I^m \right) =\varnothing , \]  the same will be true for nearby configurations (this is why we use $[1/4,3/4]$ instead of $\I$, otherwise new points could appear and hit the forbidden regions immediately).
	
	\vspace{.5em}
	
	We will next construct a $k$-fold semi-simplicial map 
	\[f_\bullet \colon X_{\bullet,\ldots,\bullet} \lra B^{E_k}_{\bullet,\cdots,\bullet}(C^{[r]}(\bI^k \times \bI^m))\] and prove its thick geometric realization is a weak equivalence. The map $f_{p_1,\ldots,p_k}$ is defined on a $(p_1,\ldots,p_k)$-simplex $(\{t^j_i\},\xi)$ by deleting from a configuration $\xi \in C^{[r]}\left((\bR^k,\bR^k \setminus \bI^k) \times \I^m\right)$ those points outside $\prod_{j=1}^k [t^j_0,t^j_{p_j}]$, and interpreting the remaining configuration as an element of $C^{[r]}\left(\bI^k \times \I^m\right)$. Since we are taking thick geometric realizations, to prove $||f_\bullet||$ is a weak equivalence, it suffices to prove each $f_{p_1,\cdots,p_k}$ is a weak homotopy equivalence.
	
	To prove this, we first observe that the inclusion of the subspace $X'_{p_1,\ldots,p_k}$ of $X_{p_1,\ldots,p_k}$ of those $(\{t^j_i\},\xi)$ such that for all $1 \leq j \leq k$ we have
	\[\xi \cap \left([1/4,3/4]^{j-1} \times ([1/4,t^j_0] \cup [t^j_{p_j},3/4]) \times [1/4,3/4]^{k-j} \times \bI^m \right) = \varnothing,\]
	is a homotopy equivalence. In other words, in $X'_{p_1,\cdots,p_k}$ all points in $\xi$ lie either in $\smash{\prod_{j=1}^k} [t^j_0,t^j_{p_j}]$ or have one of their first $k$ coordinates $<1/4$ or $>3/4$.
	
	Thus it suffices to prove that the inclusion
	\[g_{p_1,\cdots,p_k} \colon B^{E_k}_{p_1,\cdots,p_k}(C^{[r]}(\bI^k \times \bI^m)) \lra X'_{p_1,\ldots,p_k},\]
	which regards a configuration in $C^{[r]}(\bI^k \times \bI^m)$ as one in $C^{[r]}((\bR^k,\bR^k \setminus \bI^k) \times \bI^m)$, is a homotopy equivalence. Then the composition $f_{p_1,\ldots,p_k} \circ g_{p_1,\cdots,p_k}$ is the identity on $ B^{E_k}_{p_1,\cdots,p_k}(C^{[r]}(\bI^k \times \bI^m))$. A homotopy from $g_{p_1,\ldots,p_k} \circ f_{p_1,\cdots,p_k}$ to the identity on $X'_{p_1,\ldots,p_k}$ is given as follows: it is the identity on the points in $\xi$ in $\prod_{j=1}^k [t^j_0,t^j_{p_j}]$ and pushes the remaining points linearly outwards from $(1/2,\cdots,1/2)$ until all are in the regions $\bR^k \setminus \bI^k$ where they vanish.\end{proof}

\begin{remark}Snaith showed that the May-Milgram filtration stably splits \cite{Sn}. However, its lift to a filtration of $\bfE_{k+m}(\ast)$ of $\cE_k$-algebras \emph{does not} split after taking suspension spectra. Such a splitting would imply that $C_2(\I^2 \times \I^1) \simeq \bR P^2$ stably splits off a copy of $F_2^{[1]}(\I^2 \times \I^1)/\fS_2 \simeq \bR P^1$. 
	
However, this filtration \emph{does} split after stabilizing in a different manner. Recall $Q^{\cE_k}_\bL$ denotes the derived indecomposables functor of Remark \ref{rem:o-homology-ss}. Basterra-Mandell showed that derived indecomposables can be considered as stabilization of an algebra over an operad \cite{BaM05}, and derived indecomposables of an $\cE_k$-algebra may be computed using its $k$-fold bar construction, see \cite{BaM11} or Chapter 13 of \cite{GKRW1}. Thus the induced filtration of the stabilization $Q^{\cE_k}_{\bL}(\Sigma^\infty \bfE_{k+m}(\ast)_+)$ agrees with the suspension spectra of the May-Milgram filtration and hence splits by the work of Snaith \cite{Sn}.
\end{remark}

\bibliographystyle{amsalpha}
\bibliography{cellfree}

\def\cprime{$'$}
\providecommand{\bysame}{\leavevmode\hbox to3em{\hrulefill}\thinspace}
\providecommand{\MR}{\relax\ifhmode\unskip\space\fi MR }
\providecommand{\MRhref}[2]{%
  \href{http://www.ams.org/mathscinet-getitem?mr=#1}{#2}
}
\providecommand{\href}[2]{#2}
\begin{thebibliography}{GKRW19}

\bibitem[BM05]{BaM05}
M.~Basterra and M.~A. Mandell, \emph{Homology and cohomology of {$E_\infty$}
  ring spectra}, Math. Z. \textbf{249} (2005), no.~4, 903--944. \MR{2126222}

\bibitem[BM11]{BaM11}
Maria Basterra and Michael~A. Mandell, \emph{Homology of {$E_n$} ring spectra
  and iterated {$THH$}}, Algebr. Geom. Topol. \textbf{11} (2011), no.~2,
  939--981. \MR{2782549}

\bibitem[B{\"o}d87]{B}
C.-F. B{\"o}digheimer, \emph{Stable splittings of mapping spaces}, Algebraic
  topology ({S}eattle, {W}ash., 1985), Lecture Notes in Math., vol. 1286,
  Springer, Berlin, 1987, pp.~174--187. \MR{922926 (89c:55011)}

\bibitem[CCKN83]{cohencohenkuhnneisendorfer}
F.~R. Cohen, R.~L. Cohen, N.~J. Kuhn, and Joseph~A. Neisendorfer, \emph{Bundles
  over configuration spaces}, Pacific J. Math. \textbf{104} (1983), no.~1,
  47--54. \MR{683727}

\bibitem[CMM78]{cohenmahowaldmilgram}
F.~R. Cohen, M.~E. Mahowald, and R.~J. Milgram, \emph{The stable decomposition
  for the double loop space of a sphere}, Algebraic and geometric topology
  ({P}roc. {S}ympos. {P}ure {M}ath., {S}tanford {U}niv., {S}tanford, {C}alif.,
  1976), {P}art 2, Proc. Sympos. Pure Math., XXXII, Amer. Math. Soc.,
  Providence, R.I., 1978, pp.~225--228. \MR{520543}

\bibitem[dSP13]{pazzis}
C.~de~Seguins~Pazzis, \emph{The geometric realization of a simplicial
  {H}ausdorff space is {H}ausdorff}, Topology Appl. \textbf{160} (2013),
  no.~13, 1621--1632. \MR{3091338}

\bibitem[Dun88]{Dunn}
G.~Dunn, \emph{Tensor product of operads and iterated loop spaces}, J. Pure
  Appl. Algebra \textbf{50} (1988), no.~3, 237--258. \MR{938617}

\bibitem[ERW19]{ebertrw}
J.~Ebert and O.~Randal-Williams, \emph{Semisimplicial spaces}, Algebr. Geom.
  Topol. \textbf{19} (2019), no.~4, 2099--2150. \MR{3995026}

\bibitem[GKRW18]{GKRW1}
S.~Galatius, A.~Kupers, and O.~Randal-Williams, \emph{Cellular
  ${E}_k$-algebras}, \url{https://arxiv.org/abs/1805.07184}, 2018.

\bibitem[GKRW19]{GKRW2}
\bysame, \emph{${E}_2$-cells and mapping class groups}, Publications
  math{\'e}matiques de l'IH{\'E}S \textbf{130} (2019), no.~1, 1--61.

\bibitem[Hir03]{hirschhorn}
P.~S. Hirschhorn, \emph{Model categories and their localizations}, Mathematical
  Surveys and Monographs, vol.~99, American Mathematical Society, Providence,
  RI, 2003. \MR{1944041}

\bibitem[Jam55]{james}
I.~M. James, \emph{Reduced product spaces}, Ann. of Math. (2) \textbf{62}
  (1955), 170--197. \MR{0073181}

\bibitem[KM18]{KM4}
A.~Kupers and J.~Miller, \emph{{$E_n$}-cell attachments and a local-to-global
  principle for homological stability}, Math. Ann. \textbf{370} (2018),
  no.~1-2, 209--269. \MR{3747486}

\bibitem[Lur17]{lurieha}
J.~Lurie, \emph{Higher algebra}, 2017, September 2017 version,
  \url{http://www.math.harvard.edu/~lurie/papers/HA.pdf}.

\bibitem[May72]{M}
J.~P. May, \emph{The geometry of iterated loop spaces}, Springer-Verlag,
  Berlin-New York, 1972, Lecture Notes in Mathematics, Vol. 271. \MR{0420610
  (54 \#8623b)}

\bibitem[Mil66]{Mil2}
R.~J. Milgram, \emph{Iterated loop spaces}, Ann. of Math. (2) \textbf{84}
  (1966), 386--403. \MR{0206951}

\bibitem[Sna74]{Sn}
V.~P. Snaith, \emph{A stable decomposition of {$\Omega ^{n}S^{n}X$}}, J. London
  Math. Soc. (2) \textbf{7} (1974), 577--583. \MR{0339155 (49 \#3918)}

\bibitem[Str09]{stricklandcgwh}
N.P. Strickland, \emph{The category of {CGWH} spaces},
  \url{https://neil-strickland.staff.shef.ac.uk/courses/homotopy/cgwh.pdf},
  2009.

\bibitem[Wei05]{weissclassify}
M.~Weiss, \emph{What does the classifying space of a category classify?},
  Homology Homotopy Appl. \textbf{7} (2005), no.~1, 185--195. \MR{2175298
  (2007d:57059)}

\end{thebibliography}

\vspace{1em}

\end{document}